\documentclass[11pt]{amsart}
\usepackage[table]{xcolor}

\usepackage[T1]{fontenc}
\usepackage{lmodern}
\usepackage[english]{babel} 
\usepackage{amsmath}
\usepackage{amssymb}
\usepackage{mathrsfs}
\usepackage{bookmark}
\usepackage{pgfplots}
\pgfplotsset{compat=1.17}
\usetikzlibrary{decorations.markings}
\usepackage{todonotes}
\presetkeys{todonotes}{size=\tiny}{} 
\usepackage[numbers,sort&compress]{natbib}
\numberwithin{equation}{section}
\usepackage{amsthm}
\usepackage[a4paper]{geometry}
\newcommand\Item[1][]{%
  \ifx\relax#1\relax  \item \else \item[#1] \fi
  \abovedisplayskip=0pt\abovedisplayshortskip=0pt~\vspace*{-\baselineskip}}

\theoremstyle{plain}
 \newtheorem{theorem}{Theorem}[section]
 
 \newtheorem{proposition}[theorem]{Proposition}
 \newtheorem{lemma}[theorem]{Lemma}

\theoremstyle{definition}
 \newtheorem{example}[theorem]{Example}
 
 \newtheorem{remark}[theorem]{Remark}
 \numberwithin{equation}{section}

\usepackage{caption}
\usepackage{tcolorbox}

\DeclareMathOperator{\WL}{WL}

\usepackage{mathtools}

\usepackage{pinlabel}

\definecolor{mygreen}{rgb}{0.0, 0.5, 0.0}
\definecolor{deepmagenta}{rgb}{0.8, 0.0, 0.8}
\definecolor{deepcarrotorange}{rgb}{0.91, 0.41, 0.17}
\definecolor{heartgold}{rgb}{0.5, 0.5, 0.0}
\definecolor{babyblue}{rgb}{0.54, 0.81, 0.94}
\definecolor{arylideyellow}{rgb}{0.91, 0.84, 0.42}
\definecolor{apricot}{rgb}{0.98, 0.81, 0.69}
\definecolor{cambridgeblue}{rgb}{0.64, 0.76, 0.68}
\definecolor{brightturquoise}{rgb}{0.03, 0.91, 0.87}
\definecolor{capri}{rgb}{0.0, 0.75, 1.0}
\definecolor{celadon}{rgb}{0.67, 0.88, 0.69}
\definecolor{chartreuse}{rgb}{0.87, 1.0, 0.0}
\newcommand{\xuparrow}[1]{%
  {\left\uparrow\vbox to #1{}\right.\kern-\nulldelimiterspace}
}
\definecolor{kellygreen}{rgb}{0.3, 0.73, 0.09}

\usepackage{subfigure}
\usepackage{BOONDOX-cal}
\usepackage{enumitem}
\usepackage{appendix}
\graphicspath{{images/}}
\newcommand{\code}{\texttt}

\usepackage{graphicx} 
\graphicspath{{figures/}} 

\usepackage{pinlabel} 
\usepackage{tikz}

\title{Self-intersections of arcs on a pair of pants}

\author{Nhat Minh Doan, Hanh Vo}
\address{Department of Mathematics, National University of Singapore, Singapore \& Institute of Mathematics, Vietnam Academy of Science and Technology, Hanoi, Vietnam}
\email{dnminh@math.ac.vn}

\address{School of Mathematical and Statistical Sciences, Arizona State University}

\email{thihanhv@asu.edu}


\usepackage{algorithm}
\usepackage[noend]{algpseudocode}

\usepackage{longtable}
\begin{document}
\begin{abstract}  
We investigate arcs on a pair of pants and present an algorithm to compute the self-intersection number of an arc. Additionally, we establish bounds for the self-intersection number in terms of the word length. We also prove that the spectrum of self-intersection numbers of 2-low-lying arcs covers all natural numbers.
\end{abstract}

\maketitle

\section{Introduction} 
Curves play important roles in the study of surfaces. Dehn studied curves from a topological and combinatorial perspective. He expressed closed curves as finite words in the generators of the fundamental group. He introduced the word problem (characterizing the identity element of the fundamental group), the conjugacy problem (determining conjugate elements), and others.
For a long time, mathematicians have been interested in closed curves. The relationship between lengths (combinatorial lengths or hyperbolic lengths) and self-intersection numbers of closed curves has been actively investigated (see \cite{MR3380366, MR2904905, MR2676743, basmajian2024shortest, vo2022short, MR3065183, diop2022self, arenas2024taut, erlandsson2020short}). Various algorithms exist to compute self-intersections, such as those found in \cite{MR0895629,MR4040285,MR1407382,MR3413189, rickards2021computing,birman1984algorithm}.

We are interested in how the results concerning lengths and self-intersections of closed curves can be reproduced and adapted for \textit{arcs}. 
Let $S$ be a surface with boundary or punctures.
We define an infinite arc to be an immersion of the open interval $(0, 1)$ into $S$ such that the endpoints are in the set of punctures. We consider arcs up to homotopy relative to punctures, and we assume that they are not homotopic into the punctures. We have similar definitions for arcs whose at least one of its endpoints lies on a boundary component. For example, a compact arc on $S$ is a continuous map $\gamma: [0, 1]\to S$ such that $\gamma(0), \gamma(1) \in \partial S$ and $\gamma(0,1) \subset S^o$. We also identify an arc with its image and consider arcs up to homotopy relative to $\partial S$, where we allow the endpoints to move along $\partial S$. We assume that they are not homotopic into the boundary. The self-intersection number of an arc is minimal within its homotopy class. 

There are various methods to encode the set (or a subset) of arcs on surfaces, such as using continued fractions or left-right cutting sequences as described in \cite{series1985modular, labourie2018probabilistic}, or using cutting sequences with sides of an ideal triangulation of the surface, or Farey sequences as detailed in \cite{doan2021orthotree}. In this paper, we introduce a different encoding method for arcs that facilitates the computation of the number of self-intersections. We specifically focus on the set of arcs on a pair of pants (a thrice-punctured sphere), a fundamental case that can be extended to more general surfaces with punctures. We define the surface word as $a1A3b2B3$, illustrated in Figure \ref{WordSurfaceAB}. Here, $1$, $2$, and $3$ denote the boundary components (or punctures), while $a$ and $b$ represent the seams connecting punctures $1$-$3$ and $2$-$3$, respectively. In this notation, $A=a^{-1}$ and $B=b^{-1}$ are the inverses of $a$ and $b$.

\begin{figure}[h]
    \centering
    \begin{tikzpicture}
        \node[anchor=south west,inner sep=0] (image) at (0,0) {\includegraphics[width=0.5\textwidth]{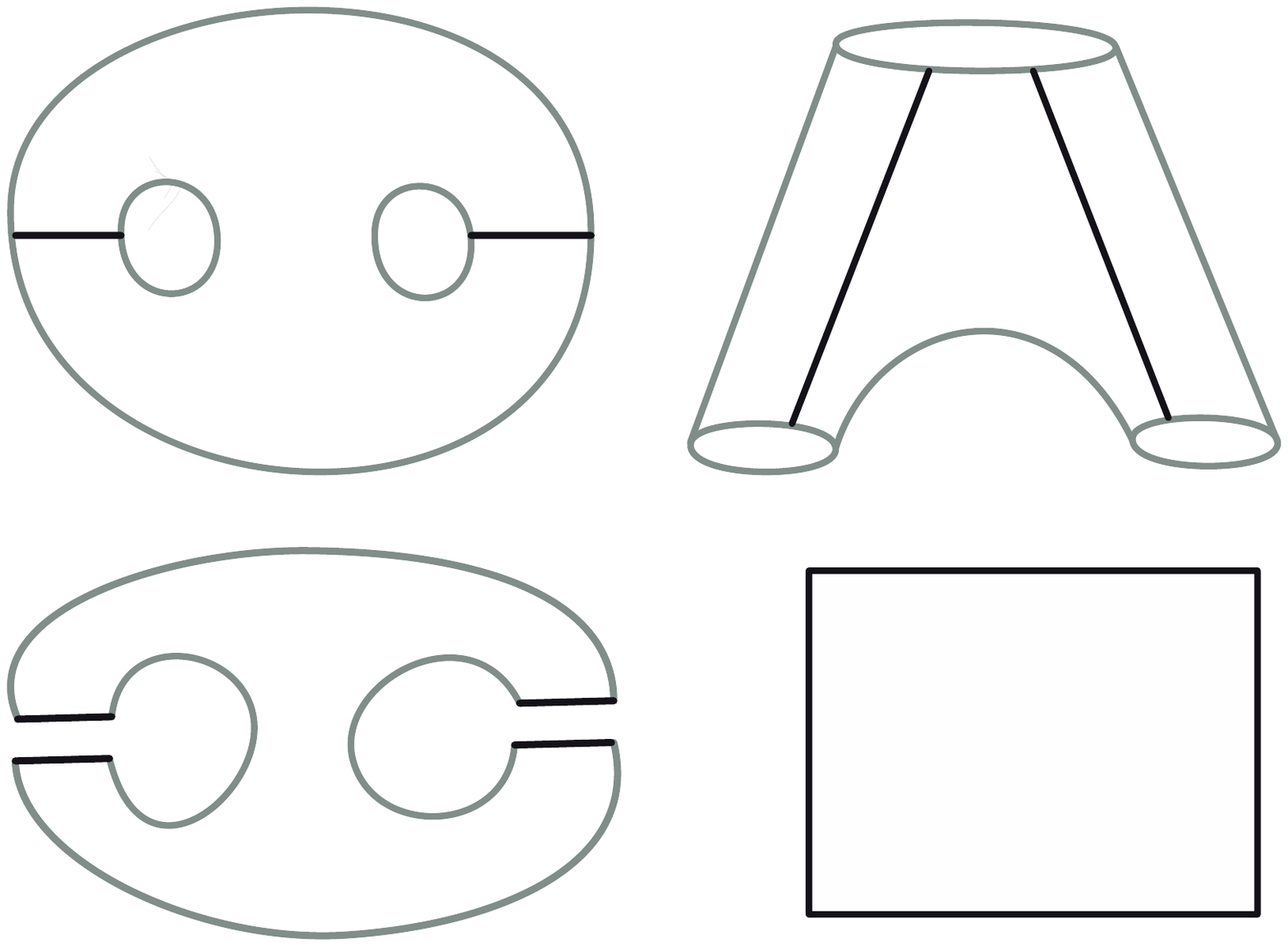}};
        \begin{scope}[x={(image.south east)},y={(image.north west)}]
            \node at (0.09,0.78) {$A$};
            \node at (0.09,0.7) {$a$};
            \node at (0.41,0.78) {$b$};
            \node at (0.41,0.7) {$B$};

            \node at (0.16,0.74) {$1$};
            \node at (0.34,0.74) {$2$};
            \node at (0.22,0.92) {$3$}; 

            \node at (0.09,0.29) {$A$};
            \node at (0.09,0.18) {$a$};
            \node at (0.44,0.30) {$b$};
            \node at (0.44,0.19) {$B$};

            \node at (0.16,0.24) {$1$};
            \node at (0.34,0.24) {$2$};
            \node at (0.23,0.37) {$3$};
            
            \node at (0.78,0.37) {$A$};
            \node at (0.64,0.22) {$a$};
            \node at (0.91,0.22) {$b$};
            \node at (0.78,0.09) {$B$};

            \node at (0.64,0.37) {$1$};
            \node at (0.91,0.09) {$2$};
            \node at (0.91,0.37) {$3$};
            \node at (0.64,0.09) {$3$};

            \node at (0.79,0.78) {$B$};
            \node at (0.84,0.78) {$b$};
            \node at (0.64,0.78) {$A$};
            \node at (0.69,0.76) {$a$};
            \node at (0.58,0.59) {$1$};
            \node at (0.90,0.59) {$2$};
            \node at (0.75,0.86) {$3$};
        \end{scope}
    \end{tikzpicture}
    \caption{The surface word for pairs of pants}
    \label{WordSurfaceAB} 
\end{figure}
 
An arc on a pair of pants corresponds to a  word $w=n_1n_2\notin\{11,22\}$ or $w=n_1 x_1 \dots x_k n_2$ 
where
\begin{itemize}
\item $n_1, n_2 \in \{1,2,3\}$,
\item $x_i \in \{a,A,b,B\}$ for $1\le i \le k$,
\item $x_i \ne x_{i+1}^{-1}$ for $1\le i < k$,
\item if $n_1=1$ then $x_1 \notin \{a,A\}$,
\item if $n_1=2$ then $x_1 \notin \{b,B\}$,
\item if $n_2=1$ then $x_k \notin \{a,A\}$,
\item if $n_2=2$ then $x_k \notin \{b,B\}$. 
\end{itemize}

Note that there are no restrictions when $n_1=3$ or $n_2=3$; in these cases, the arc will necessarily pass through the seam connecting punctures $1$ and $2$. The word length of $w$ is defined as the number of letters in the word. We denote the word length of $w$ by $\WL(w)$ and the self-intersection number of $w$ by $i(w)$. Our results focus on the topological quantities of arcs, including word lengths, self-intersection numbers and their relations. For a geometric perspective, we refer to \cite{basmajian1993orthogonal,basmajian2024orthosystoles, basmajian2020prime, doan2021orthotree,parlier2020geodesic,trin2023counting}, which discuss orthogeodesics.

We start by explaining an algorithm to compute the self-intersection number of an arc. This algorithm builds upon techniques developed for computing self-intersection numbers of closed curves on surfaces, specifically motivated from the work of Cohen-Lustig \cite{MR0895629} and Despré-Lazarus \cite{MR4040285}.

\begin{theorem}\label{algorithm}
The self-intersection of an arc $w=n_1x_1\dots x_{L}n_2$ can be computed in $O(L^2)$.
\end{theorem}   

We also bound the self-intersection number of an arc by its word length.
\begin{theorem}\label{theorem lower bound intersections}
Let $w=n_1x_1\dots x_Ln_2$. Then  
$i(w)\le \frac{L(L+1)}{2}$
and
\[
i(w) \ge 
\begin{cases}
\frac{L}{2} - 1 & \textup{ if $L$ is even}\\
\frac{L-1}{2} & \textup{ if $L$ is odd}
\end{cases}
\]
The lower bounds are achieved by $1(BA)^{L/2}2$ if $L$ is even and $1(ba)^{(L-1)/2}b1$ if $L$ is odd.
\end{theorem}

Based on computational experiments (see Table \ref{table:min_max_intersections}), 
we expect the maximal self-intersection number of $w$ to be $\frac{L^2}{4}+L$ if $L$ is even and to be $\frac{L^2-1}{4}+L$ if $L$ is odd, and they are attained by the self-intersection numbers of the arcs $1(bA)^n3$ and $3(bA)^nb3$.

Lastly, some recent studies have directed their attention towards counting various types of low-lying curves on the modular surface \cite{basmajian2022combinatorial, basmajian2023low, bourgain2017beyond} or in more general contexts \cite{erlandsson2024counting}, as well as exploring intriguing identities of specific families of low-lying arcs on general hyperbolic surfaces \cite{basmajian2020prime}. We define an arc as \(k\)-\textit{low-lying} if it wraps around the same cuff (a boundary component) consecutively at most \(k\) times. 
Our computations reveal that for any natural number \(N \leq 1000\), there exists a \(2\)-low-lying arc whose self-intersection number is \(N\). This is illustrated in Table \ref{table_low_lying_with_i} for \(N \leq 89\). Moreover, we prove that this property holds for all \(N \in \mathbb{N}\).

\begin{theorem}\label{theorem low lying}
    The spectrum of self-intersection numbers of 2-low-lying arcs on a pair of pants covers all natural numbers.
\end{theorem}

We remark that low-lying arcs are related to fractions with bounded entries (or partial quotients) in their continued fractions \cite{series1985modular}. Zaremba's conjecture \cite{zaremba1972methode,bourgain2014zaremba} predicts that the denominators of these fractions cover all natural numbers, as detailed in Section \ref{sec:lowlying}. The geometric interpretation of the denominator of these fractions is the lambda length of the associated arcs \cite{penner1987decorated}. Our result above addresses a topological aspect of this problem—specifically, the self-intersection numbers of arcs rather than the lambda lengths—and, as we will see in the final section, the proof is straightforward by carefully choosing some special family of low-lying arcs.

\subsection*{Acknowledgements}
We thank Moira Chas, Vincent Despr\'e, Chris Leininger, Hugo Parlier, and Ser Peow Tan for helpful conversations. The first name author is supported by NRF E-146-00-0029-01, and NAFOSTED 101.04-2023.33.

\section{Computing self-intersection numbers}  

\begin{proof}[Proof of Theorem \ref{algorithm}]
To compute the self-intersection numbers of $w$, we lift it to the planar model. 
A lift of $w$ is made of $L+1$ segments, namely $w_1=n_1x_1$,  $w_i=x_{i-1}^{-1}x_{i}$ for $i=2,\dots,L$ and $w_{L+1}=x_{L}n_2$ (see Figure \ref{OneLift}).  
Figure \ref{wi position} illustrates all possible configurations of a segment $w_i$ in a fundamental domain.  
A pair $(w_i, w_j)$ is said to be \textit{decidable} if we can decide if there is an intersection point from the current fundamental domain. More precisely, $w_i$ and $w_j$ are decidable if: 
\begin{itemize}
    \item an endpoint of $w_i$ and an endpoint of $w_j$ are in the same corner of the fundamental domain, that is, in a puncture $(1, 2, 3)$, or
    \item their endpoints are on different edges ($a,A,b,B$) or corners of the fundamental domain. 
\end{itemize}
It is \textit{undecidable} otherwise, meaning that at least one endpoint of $w_i$ and one endpoint of $w_j$ lie on the same edge of the fundamental domain ($a,A,b,B$). For example, $(aB, Ab)$ is a decidable pair because their endpoints lie on different edges of the fundamental domain, and $(ab,aB)$ is an undecidable pair. 
In Figure \ref{Decidable lifts}, we illustrate some examples of decidable and undecidable pairs $(w_i, w_j)$: the upper half for decidable and the lower half otherwise.
The upper left half of Figure \ref{Decidable lifts} represents an example where there is no intersection, which we shall refer to as an \textit{non-intersecting} decidable pair, while the upper right half represents an example where there is an intersection, which we shall refer to as an \textit{intersecting} decidable one. 
We notice that in the first case of decidability, namely if an endpoint of $w_i$ and an endpoint of $w_j$ are in the same corner of the fundamental domain, that is, in a puncture $( 1,2, 3)$, the pair $(w_i,w_j)$ is non-intersecting. 
The full list of intersecting and non-intersecting decidable pairs $(w_i,w_j)$ is given as Tables \ref{tab:w_i_w_j Intersecting} and \ref{tab:w_i_w_j_2 Non-intersecting}.

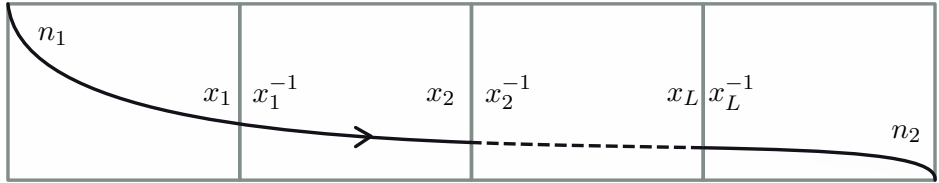
\begin{figure}
    \centering
    \begin{tikzpicture}
        \node[anchor=south west,inner sep=0] (image) at (0,0) {\includegraphics[width=0.75\textwidth]{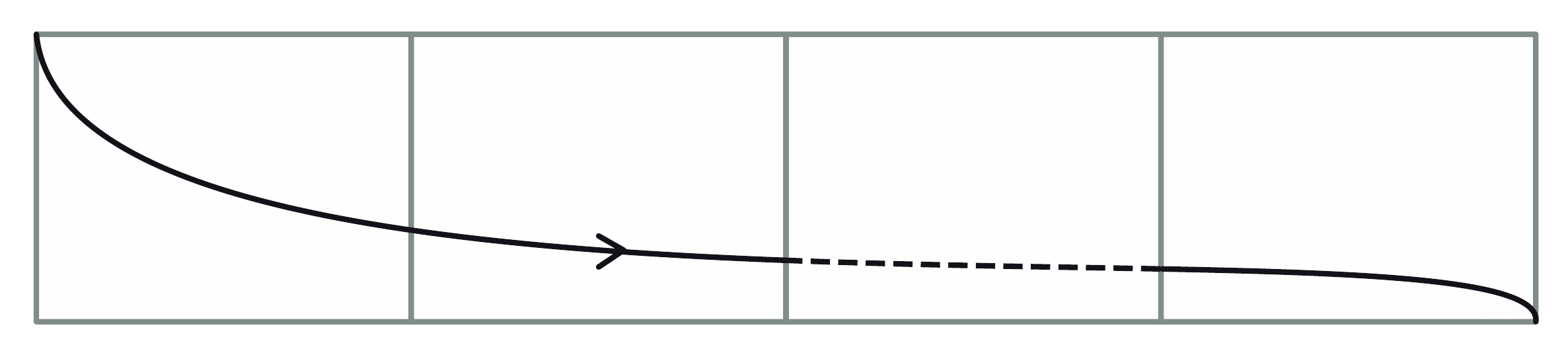}};
        \begin{scope}[x={(image.south east)},y={(image.north west)}]
            \node at (0.07,0.75) {$n_1$}; 
            \node at (0.95,0.30) {$n_2$};

            \node at (0.24,0.47) {$x_1$};
            \node at (0.30,0.49) {$x_1^{-1}$};
            \node at (0.47,0.47) {$x_2$}; 
            \node at (0.54,0.49) {$x_2^{-1}$};

            \node at (0.72,0.47) {$x_L$};
            
            \node at (0.77,0.49) {$x_L^{-1}$};
        \end{scope}
    \end{tikzpicture} 
    \caption{A lift of $w=n_1x_1\dots x_{L}n_2$ to the planar model}
    \label{OneLift} 
\end{figure}

\begin{figure}[h]
    \centering
    \begin{tikzpicture}
        \node[anchor=south west,inner sep=0] (image) at (0,0) {\includegraphics[width=0.95\textwidth]{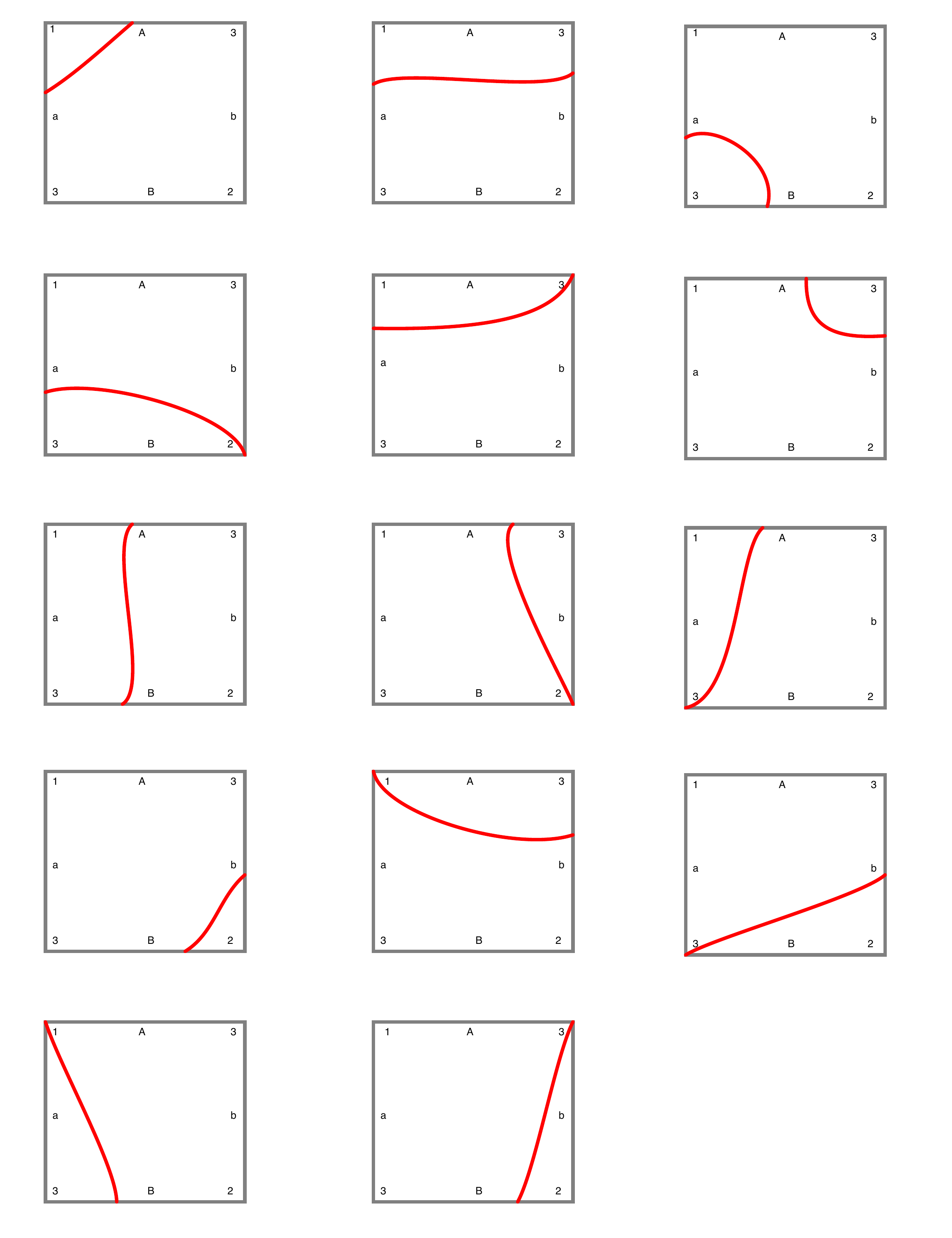}};
        \begin{scope}[x={(image.south east)},y={(image.north west)}]
        \end{scope}
    \end{tikzpicture}
    \caption{All possible configuration of a segment $w_i$ in a fundamental domain}
    \label{wi position} 
\end{figure}

\begin{figure}[h]
    \centering
    \begin{tikzpicture}
        \node[anchor=south west,inner sep=0] (image) at (0,0) {\includegraphics[width=0.5\textwidth]{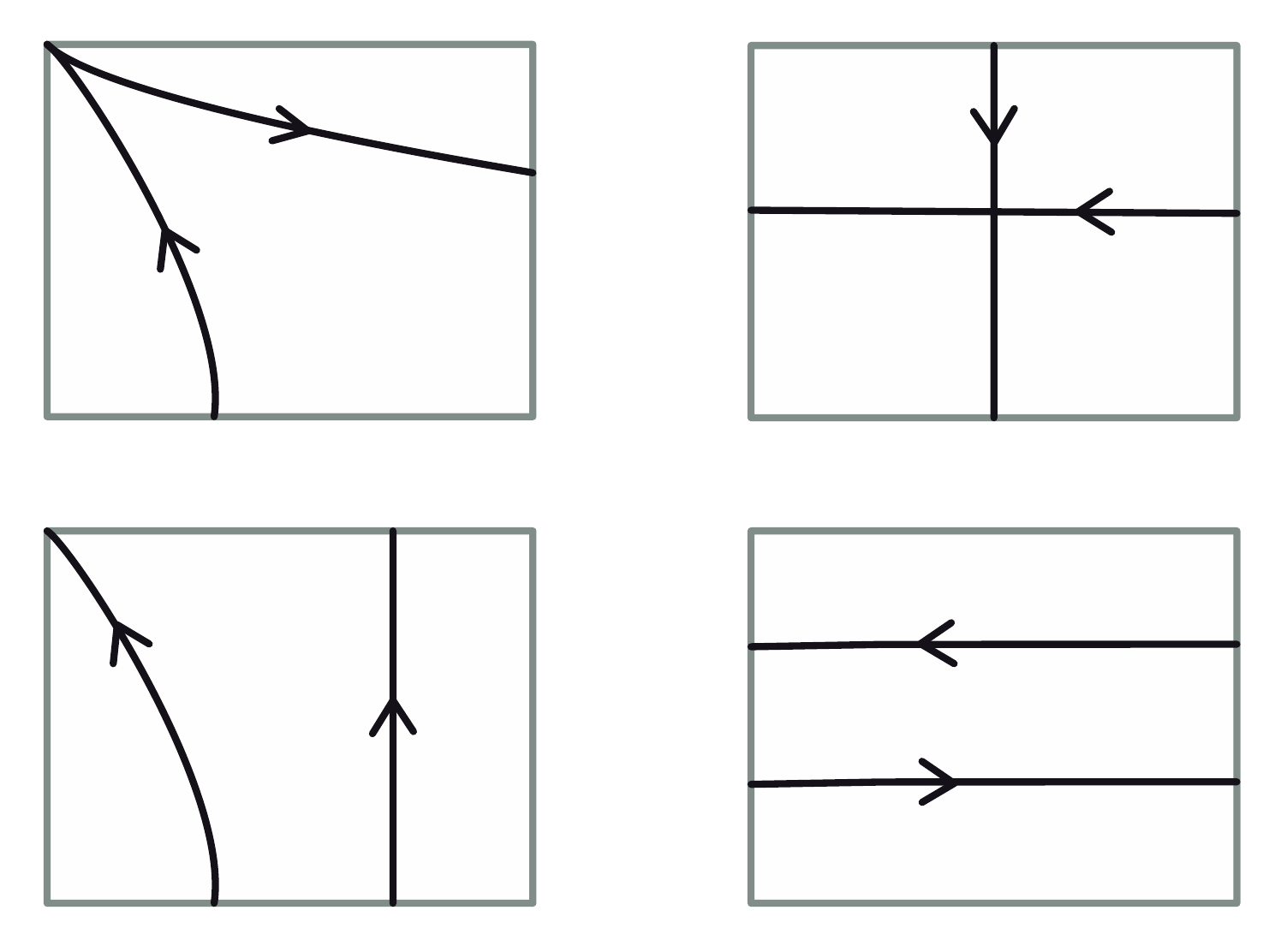}};
        \begin{scope}[x={(image.south east)},y={(image.north west)}]
        \end{scope}
    \end{tikzpicture}
    \caption{Examples: upper half - decidable pairs, lower half - undecidable pairs}
    \label{Decidable lifts} 
\end{figure}

\begin{table}[h]
\medskip
\renewcommand{\arraystretch}{1.5}
\centering
\begin{tabular}{|l|l|}
\hline
\( w_i \) & \( w_j \) \\
\hline
\( aA, Aa \) & \( 1b, b1, 1B, B1 \) \\
\hline
\( ab, ba \) & \( AB, BA, 2A, A2, 3A, A3, 1B, B1, 3B, B3 \) \\
\hline
\( aB, Ba \) & \( 3A, A3, 3b, b3 \) \\
\hline
\( 2a, a2 \) & \( AB, BA, 3A, A3,bB, Bb, 3b, b3, 1B,B1, 3B, B3 \) \\
\hline
\( 3a, a3 \) & \( Ab, bA, AB, BA, 2A, A2, 3A, A3, 1b, b1, 1B, B1 \) \\
\hline
\( Ab, bA \) & \( 3B, B3 \) \\
\hline
\( AB, BA \) & \( ab, ba, 2a, a2, 3a, a3, 1b, b1, 3b, b3 \) \\
\hline
\( 2A, A2 \) & \( ab, ba, 3a, a3, bB, Bb, 1b, b1, 3b, b3, 3B, B3 \) \\
\hline
\( 3A, A3 \) & \( ab, ba, aB, Ba, a2, 2a, a3, 3a, b1, 1b, B1, 1B \) \\
\hline
\( bB, Bb \) & \( a2, 2a, A2, 2A \) \\
\hline
\( 1b, b1 \) & \( aA, Aa, a3, 3a, AB, BA, A2, 2A, A3, 3A, B3, 3B \) \\
\hline
\( 3b, b3 \) & \( aB, Ba, a2, 2a, AB, BA, A2, 2A, B1, 1B, B3, 3B \) \\
\hline
\( 1B, B1 \) & \( aA, Aa, ab, ba, a2, 2a, a3, 3a, A3, 3A, b3, 3b \) \\
\hline
\( 3B, B3 \) & \( ab, ba, a2, 2a, Ab, bA, A2, 2A, b1, 1b, b3, 3b \) \\
\hline
\end{tabular}
\caption{Intersecting decidable pairs \( (w_i,w_j) \)}
\label{tab:w_i_w_j Intersecting}
\end{table}

\begin{table}[h]
\medskip
\renewcommand{\arraystretch}{1.5}
\centering
\begin{tabular}{|l|l|}
\hline
\( w_i \) & \( w_j \) \\
\hline
\( aA, Aa \) & \( bB, Bb, b3, 3b, B3, 3B \) \\\hline
\( aB, Ba \) & \( Ab, bA, A2, 2A, b1, 1b \) \\\hline
\( a2, 2a \) & \( Ab, bA, A2, 2A, b1, 1b \) \\\hline
\( a3, 3a \) & \( bB, Bb, b3, 3b, B3, 3B \) \\\hline
\( Ab, bA \) & \( aB, Ba, a2, 2a, B1, 1B \) \\\hline
\( A2, 2A \) & \( aB, Ba, a2, 2a, B1, 1B \) \\\hline
\( bB, Bb \) & \( aA, Aa, a3, 3a, A3, 3A \) \\\hline
\( b1, 1b \) & \( aB, Ba, a2, 2a, B1, 1B \) \\\hline
\( b3, 3b \) & \( aA, Aa, a3, 3a, A3, 3A \) \\\hline
\( B1, 1B \) & \( Ab, bA, A2, 2A, b1, 1b \) \\\hline
\( B3, 3B \) & \( aA, Aa, a3, 3a, A3, 3A \) \\
\hline
\end{tabular}
\caption{Non-intersecting decidable pairs \( (w_i,w_j) \)}
\label{tab:w_i_w_j_2 Non-intersecting}
\end{table}

We first set \code{intersection} to be $0$. 
For $i=1, \dots, L$ and $j= i+1,\dots, L+1$, if the pair $(w_i,w_j)$ is unchecked, we will check if they intersect.  
In case of decidability, if they intersect, we increase \code{intersection} by $1$ and mark it as checked. 
In the case of undecidability, namely, at least one endpoint of $w_i$ and one endpoint of $w_j$ are in the same edge of the fundamental domain, that is, $a,A,b,B$. 
In this case, we extend $w_i$ (and similar for $w_j$) by moving to the next segment in the lift of $w$, namely $w_{i+1  \mod \WL(w)}$, which we shall refer to as \textit{forwarding}, or
by moving to the previous segment in the lift of $w$, namely  $w_{i-1 \mod \WL(w)}$, which we shall refer to as \textit{backwarding}.
We keep forwarding or backwarding until obtaining a divergent pair $(w_i',w_j')$, that is, until their other endpoints lie on different edges or corners of the fundamental domain. 
This process must terminate because the extension must reach the punctures, if not earlier. 
Note that we mark all the pairs while forwarding or backwarding as checked, namely 
\[(w_{i+1  \mod \WL(w)}, w_{j+1  \mod \WL(w))},
(w_{j+1  \mod \WL(w)}, w_{i+1  \mod \WL(w))}, 
\dots, (w_i',w_j'), (w_j',w_i').
\] 
After forwarding and backwarding, we reach the configuration where we can decide whether $(w_i,w_j)$ creates an intersection. If so, we increase \code{intersection} by $1$.  

\subsubsection*{Example:}
In Figure \ref{Nobigon}, we illustrate an example of extending an undecidable pair $w_i, w_j$.  
Denote by $\code{start}(w_i)$ and $\code{end}(w_i)$ the starting and ending letter of $w_i$. In this case $\code{start}(w_i) = \code{end}(w_j)$ and $\code{end}(w_i) = \code{start}(w_j)$. In one direction, we forward $w_i$ and backward $w_j$ until divergent, we shall call them $w_i^1, w_j^1$. In the other direction, we forward $w_j$ and backward $w_i$ until divergent, we shall call them  $w_i^2, w_j^2$. There is an intersection in this case because $w_j^1$ is on the right of $w_i^1$ whereas $w_j^2$ is on the left of $w_i^2$.

\begin{figure}
    \centering
    \begin{tikzpicture}
        \node[anchor=south west,inner sep=0] (image) at (0,0) {\includegraphics[width=0.75\textwidth]{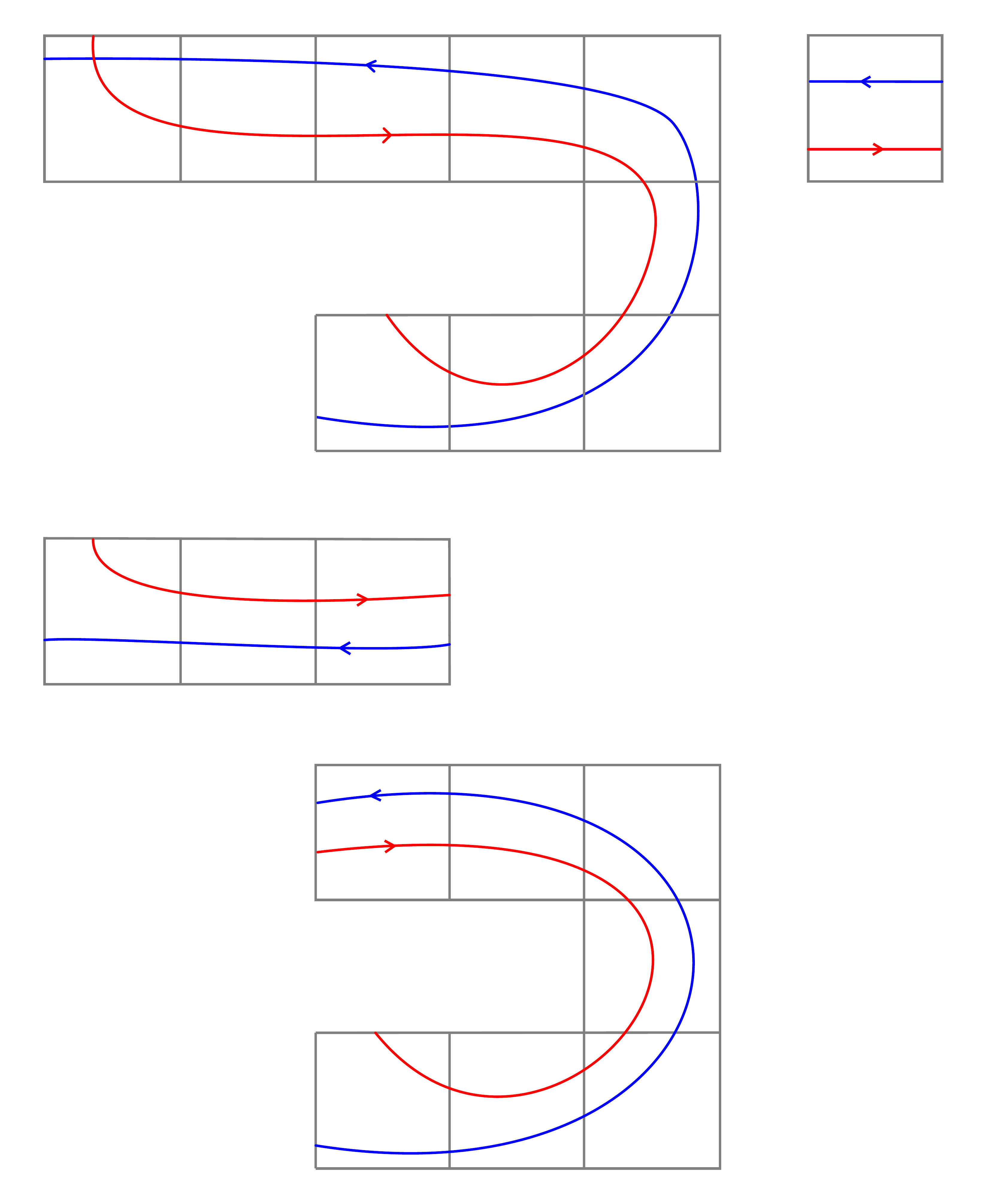}};
        \begin{scope}[x={(image.south east)},y={(image.north west)}]
            
            \node at (0.38,0.45) {$w_i$};
            \node at (0.38,0.52) {$w_j$};
            \node at (0.36,0.35) {$w_i$}; 
            \node at (0.37,0.28) {$w_j$};

            \node at (0.38,0.96) {$w_i$}; 
            \node at (0.38,0.87) {$w_j$}; 

            \node at (0.08,0.45) {$w_i^1$}; 
            \node at (0.08,0.52) {$w_j^1$}; 

            \node at (0.37,0.06) {$w_i^2$}; 
            \node at (0.37,0.12) {$w_j^2$}; 

            \node at (0.88,0.95) {$w_i$}; 
            \node at (0.87,0.89) {$w_j$}; 
 
        \end{scope}
    \end{tikzpicture} 
    \caption{Forward and backward $w_i, w_j$}
    \label{Nobigon} 
\end{figure}

The obtained \code{intersection} number after checking all pairs $(w_i,w_j)$ is the self-intersection number of $w$.  
There are $L(L+1)/2$ pairs to check. For each pair we need to extend at most $L$ times, so there are $L$ pairs to check. Thus the self-intersection number of $w$ can be computed in at most $O(L^3)$ steps. However, as explained in \cite{MR0895629}[Corollary 19], we only need a quadratic time: the quadratic complexity comes from the fact that we store the decision while extending to be able to decide (meaning that we do not need the potentially linear exploration each time). 
\end{proof}

\begin{example}
In this example, we compute the self-intersection number of the arc $w=1BABA2$. A lift of $w$ is made of $5$ segments $w_1=1B, w_2=bA, w_3=aB, w_4=bA,w_5=a2$ (see Figure \ref{1BABA2}).
For $i=1,\dots,4$ and $j=i,\dots,5$, we check if a self-intersection of $w$ is coming from the pair $(w_i,w_j)$ (see Table \ref{1BABA2intersection}). If $i$ and $j$ have different parity, the pair $(w_i,w_j)$ is a non-intersecting decidable, so they do not contribute intersections. The pair $(w_1,w_3)=(1B,aB)$ is undecidable because both $w_1$ and $w_3$ end at $B$. We need to forward this pair to $(w_2,w_4)=(bA,bA)$, which is again undecidable so we continue forward it to $(w_3,w_5)=(aB,a2)$. Now it is decidable: at the beginning, $w_3$ is on the right of $w_1$, while at the end of the forwarding process, the corresponding lift of $w_3$ is on the left of that of $w_1$. Therefore, there is an intersection point coming from the pair $(w_1,w_3)$ (or equivalently, from the pairs $(w_2,w_4)$ and $(w_3,w_5)$). We mark all these pairs as checked. Now we only have the pair $(w_1,w_5)=(1B,2a)$ to check. This pair is an intersecting decidable pair, so it contributes one self-intersection for $w$. 
Thus the number of self-intersections of $w$ is $2$.

\begin{figure}
    \centering 
    \begin{tikzpicture}
        \node[anchor=south west,inner sep=0] (image) at (0,0) {\includegraphics[width=0.7\textwidth]{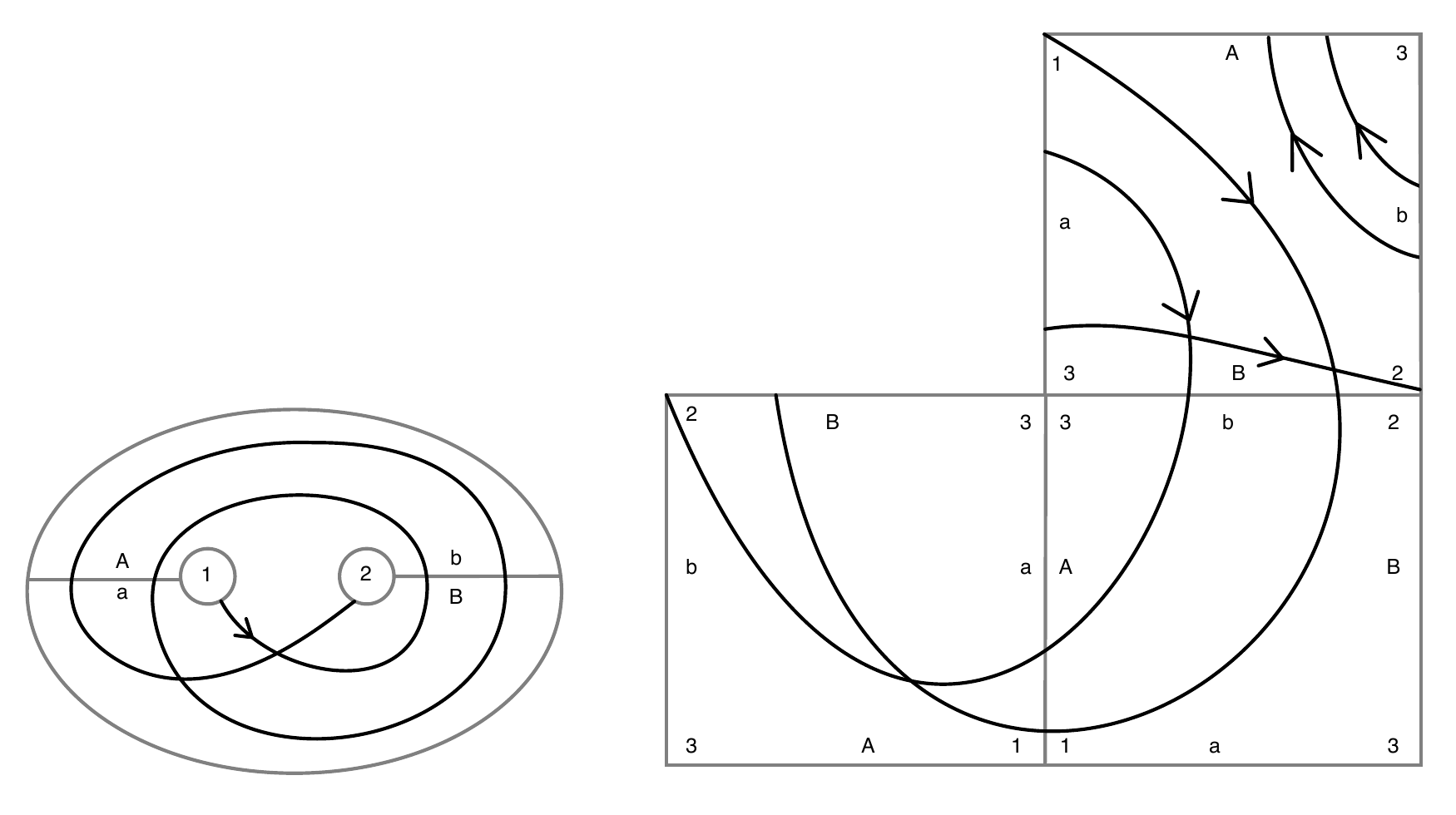}};
        \begin{scope}[x={(image.south east)},y={(image.north west)}]
            \node at (0.86,0.85) {$w_2$};
            \node at (0.82,0.76) {$w_1$}; 
            \node at (0.86,0.61) {$w_5$}; 
            \node at (0.95,0.86) {$w_4$};  
            \node at (0.77,0.64) {$w_3$};  
        \end{scope}
    \end{tikzpicture}  
    \caption{Arc $1BABA2$ and the segments $w_1,\dots,w_5$ of its lifts}
    \label{1BABA2} 
\end{figure}
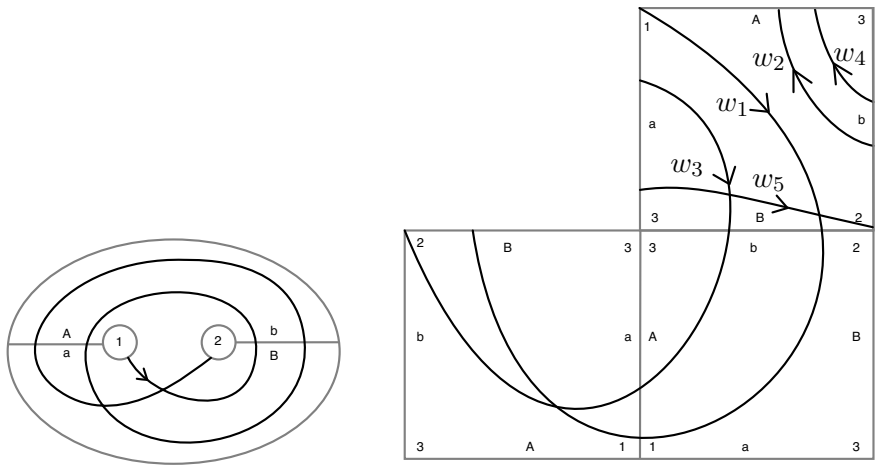

\begin{table}[h]
    \centering
    \begin{tabular}{|c|c|c|c|c|c|} 
        \hline
&$w_1=1B$&$w_2=bA$&$w_3=aB$&$w_4=bA$&$w_5=a2$\\\hline 
     $w_1=1B$& \cellcolor{gray!30} & 0& X& 0 & 1 \\\hline
        $w_2=bA$& \cellcolor{gray!30} & \cellcolor{gray!30} & 0 & X & 0 \\\hline
       $w_3=aB$&  \cellcolor{gray!30} & \cellcolor{gray!30} & \cellcolor{gray!30} & 0 & 1\\\hline
        $w_4=bA$& \cellcolor{gray!30} & \cellcolor{gray!30} & \cellcolor{gray!30} & \cellcolor{gray!30} & 0\\\hline
        $w_5=a2$&\cellcolor{gray!30}  & \cellcolor{gray!30} & \cellcolor{gray!30}  & \cellcolor{gray!30} & \cellcolor{gray!30}    \\
        \hline
    \end{tabular}
    \caption{Computing the self-intersection number of $1BABA2$}
    \label{1BABA2intersection}
\end{table}
\end{example}

\begin{remark}
Our implementation for a pair of pants is available on \url{https://github.com/hanhv/self-intersections-arcs}. 
The list of all simple arcs is 
\[\{12, 13, 21, 23, 31, 32, 33, 1b1, 1B1, 2a2, 2A2\}.\]
We also computed the self-intersection numbers of all arcs of a given word length and we saw that the distribution of self-intersection numbers is expected to be Gaussian (see Figure \ref{histogram_wordLength16_smallfont} for word length is $16$).
This agrees with the results by Chas-Lalley \cite{MR2909769} for closed curves.

\begin{figure}
    \centering    \includegraphics[width=0.8\textwidth]{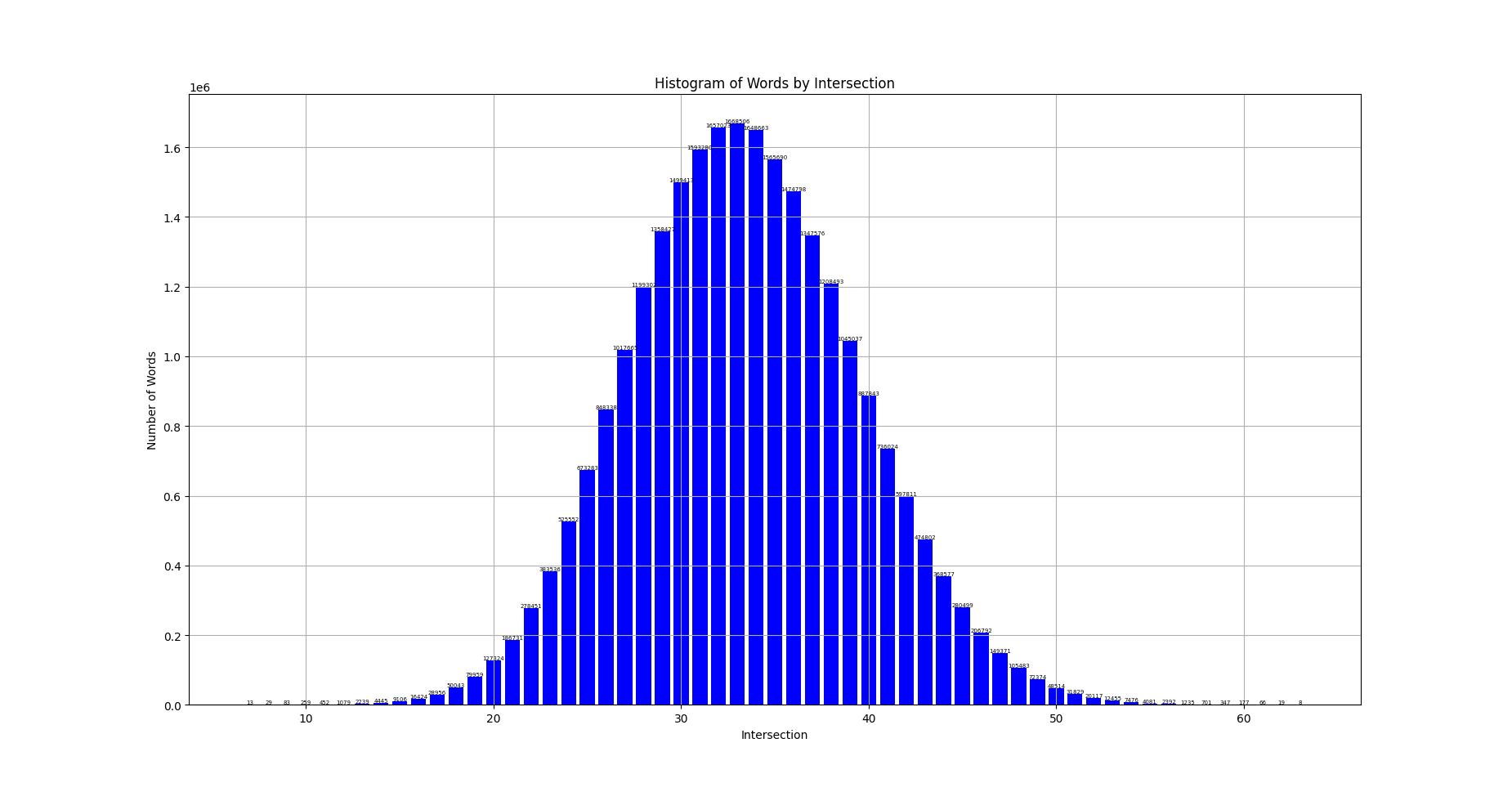}
    \caption{Self-intersection number of arcs with word length $16$}
    \label{histogram_wordLength16_smallfont} 
\end{figure}
\end{remark}

\begin{remark}
The idea of the algorithm works for arcs on other surfaces with given surface words. For instance, the surface word of the punctured torus is $ABab$. In addition, the idea also works for computing intersection numbers between two arcs. 
\end{remark}

\section{Bounding self-intersection numbers}
In this section, we bound the self-intersection number of arcs of a given word length. Table \ref{table:min_max_intersections} presents the computational results for the minimum and maximum self-intersection numbers.

\begin{table}[h!]
\centering
\begin{tabular}{|c|c|c|}
\hline
\textbf{Word length} & \textbf{Minimum} & \textbf{Maximum} \\
\hline
2 & 0 & 0 \\
\hline
3 & 0 & 1 \\
\hline
4 & 1 & 3 \\
\hline
5 & 1 & 5 \\
\hline
6 & 2 & 8 \\
\hline
7 & 2 & 11 \\
\hline
8 & 3 & 15 \\
\hline
9 & 3 & 19 \\
\hline
10 & 4 & 24 \\
\hline
11 & 4 & 29 \\
\hline 
12 &5&35\\
\hline
13&5&41\\\hline
14&6&48\\
\hline
15&6&55\\
\hline
16&7&63\\
\hline
\end{tabular}
\caption{Minimum and maximum self-intersection numbers of arcs of given word lengths} 
\label{table:min_max_intersections}
\end{table}

A word in the generators of a surface group and their inverses is \textit{positive} if no generator occurs along with its inverse. Note that a positive word is
automatically reduced.
If $w$ is a word in the alphabet $\{a, A, b, B\}$, we denote by $\alpha(w)$ (resp. $\beta(w)$) the total number of occurrences of $a$ and $A$ (resp. $b$ and $B$).  
\begin{lemma}\label{lemma positive word}
For any reduced word $w$ in the alphabet ${a, A, b, B}$,  there is a positive word $w'$ of the same word length with $\alpha(w')=\alpha(w), \beta(w')=\beta(w)$ and $i(w')\le i(w)$.
\end{lemma}

\begin{proof}
We can change $w=n_1x_1\dots x_L n_2$ into a word written with $x_i \in \{a,b\}$ for $i=1,\dots,L$ while controlling the self-intersection number. If all the letters in $x_i$ are capitals, take $w'=w^{-1}$. 
Otherwise, look in $w$ for a maximal  connected string of (one or
more) capital letters. The letters at the ends of this string must be one of the pairs
$(A, A),(A, B),(B, A),(B, B)$. In the case $(B, B)$ (the other three cases admit a similar analysis), let that string be $B^{b_1}A^{a_1}B^{b_2}A^{a_2}\dots B^{b_k}$, and $w$ is either one of the following cases
\begin{align*}  
w&=n_1B^{b_1} \dots B^{b_k}a^{a_k}y, \quad n_1\ne2,\\
w&=xa^{a_0}B^{b_1} \dots B^{b_k}n_2, \quad n_2\ne 2,\\
w&=xa^{a_0}B^{b_1} \dots B^{b_k}a^{a_{k}}y,
\end{align*}
where $x,y$ stand for the rest of the word.

For the first case, consider a representative of $w$ with minimal self-intersection. In this representative, consider the arcs corresponding to the segments $n_1B$  (joining $n_1$ to the first $B$ of $B^{b_1}$) and $ba$ (joining the last small $b$ that arises from $B^{b_k}$ to the first $a$ in $a^{a_k}$). These two arcs intersect at a point $P$. By smoothing the arc at $P$ (see Figure \ref{FigSmoothingPositiveArc}), we obtain an arc corresponding to the word
\[
w'=n_1(B^{b_1} \dots B^{b_k})^{-1}a^{a_k}y, \quad n_1\ne 2,
\]
\[ 
\begin{split}
w'&=n_1(B^{b_1} \dots B^{b_k})^{-1}a^{a_k}y, \quad n_1\ne 2,\\
 & = n_1b^{b_k}\dots b^{b_1}a^ky
\end{split}
\] 
\begin{figure}
    \centering
    \begin{tikzpicture}
        \node[anchor=south west,inner sep=0] (image) at (0,0) {\includegraphics[width=0.4\textwidth]{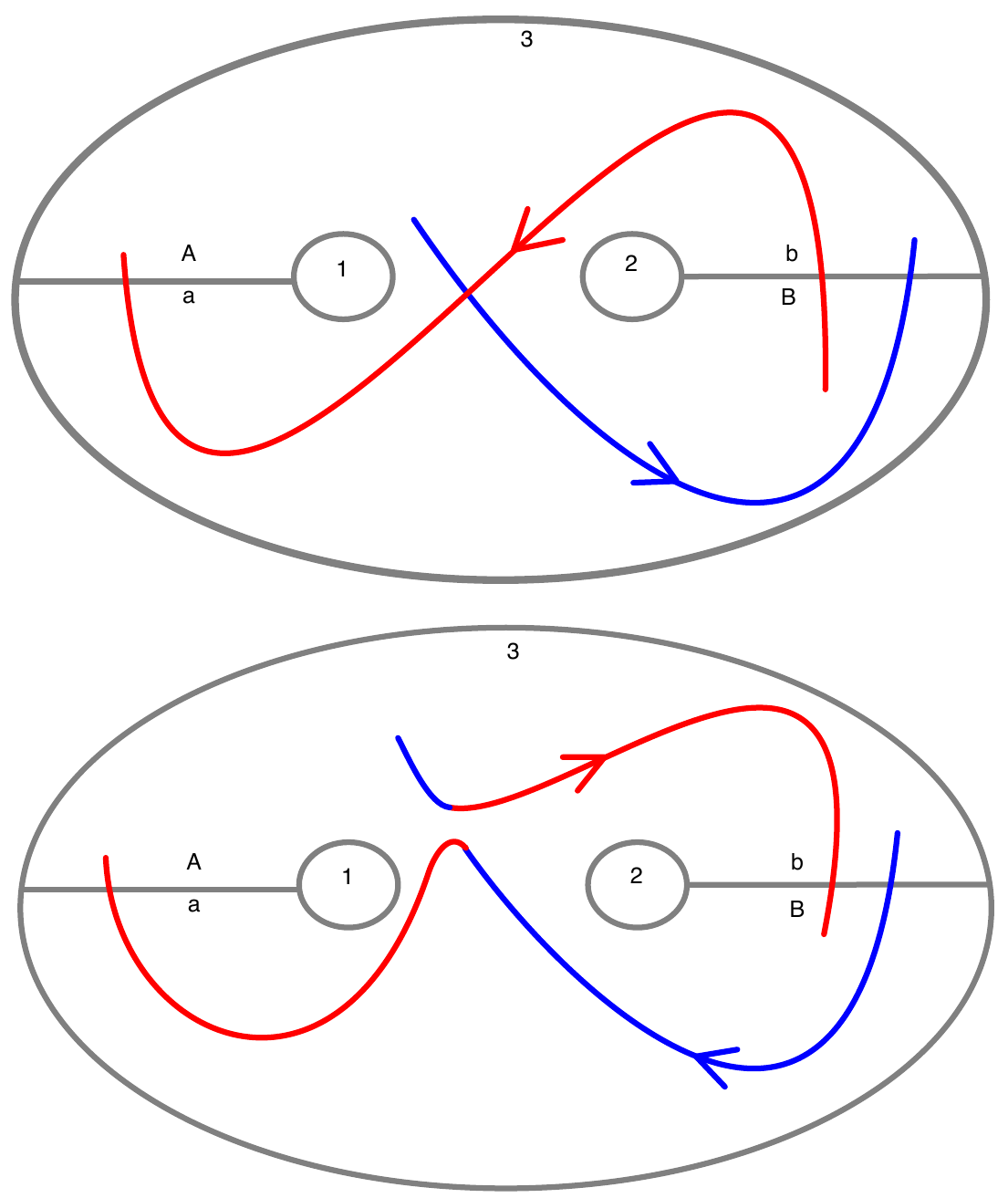}};
        \begin{scope}[x={(image.south east)},y={(image.north west)}]
            \node at (0.46,0.71) {$P$};
            \node at (0.42,0.84) {$n_1$};
            \node at (0.82,0.65) {$\vdots$}; 
            \node at (0.69,0.65) {$w_3$}; 
            \node at (0.91,0.83) {$\vdots$};  
            \node at (0.16,0.83) {$a^{a_k-1}y$};  
            \node at (0.41,0.4) {$n_1$};
            \node at (0.82,0.2) {$\vdots$}; 
            \node at (0.69,0.65) {$w_3$}; 
            \node at (0.89,0.34) {$\vdots$};  
            \node at (0.14,0.32) {$a^{a_k-1}y$};  
        \end{scope}
    \end{tikzpicture} 
    \caption{Smoothing the arc $w$ at $P$}
    \label{FigSmoothingPositiveArc} 
\end{figure}
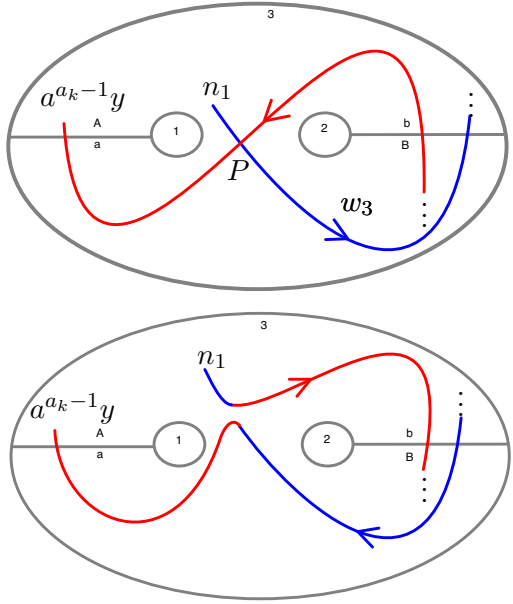

\noindent
This word has the same $\alpha$ and $\beta$ values as $w$, has lost at least one self-intersection,
and has strictly fewer upper-case letters than w. 
Similarly, by the smoothing argument, we obtain the arc
\[
w'=xa^{a_0}(B^{b_1} \dots B^{b_k})^{-1}n_2, \quad n_2\ne 2,
\]
for the second case, and
\[
w'=xa^{a_0}(B^{b_1} \dots B^{b_k})^{-1}a^{a_{k}}y
\]
for the third case.
In either case, $w'$ has the same $\alpha$ and $\beta$ values as $w$, has lost at least one self-intersection, and has strictly fewer upper-case letters than $w$. 
The process may be repeated until all upper-case letters have been eliminated.
\end{proof} 

\begin{lemma}\label{lemma concatenation}
Let $w$ be a non-simple arc. Then $w$ can be written as the concatenation $w = u \cdot v$ of an arc $u$ and a curve $v$ satisfy $i(u) + i(v) + 1 \le i(w)$. 
\end{lemma}

\begin{proof}
Consider a minimal representative of $w$. It must have self-intersections. 
Write $w = x_1\dots x_L$, where $x_1,x_n\in\{1,2,3\}$ and $x_2,\dots,x_{L-1}\in\{a,A,b,B\}$. 
(This is the only time we write $w = x_1\dots x_L$ instead of $n_1x_1\dots x_Ln_2$). Suppose that $x_ix_{i+1}$ and $x_jx_{j+1}$, with $i<j$, are the two segments intersecting at $P$. By smoothing this intersection point, we obtain the arc $u = x_1\dots x_ix_{j+1} \dots x_L$ and the curve $v = x_{i+1} \dots x_j$ as in the statement of the lemma. (In case $i+1 = j$, $v$ is a single-letter word.) 
The change from $w$ to $u \cup v$ does not add any new intersections, while the intersection corresponding to $P$ is lost. Hence $i(u) + i(v) + 1 \le i(w)$. 
\end{proof}

The following proposition is a result for positive words representing closed curves. 
\begin{proposition}\label{prop intersection and alpha}
If $w$ is a positive word representing an arc then
$i(w) \ge \alpha (w) - 1$ and $i(w) \ge \beta (w) - 1$.
\end{proposition}

\begin{proof}  
    If the word length of $w$ is at most $3$, then both $\alpha(w)$ and $\beta(w)$ are at most $1$, hence the statement is true. If word length $w$ is $L\ge4$, then $w$ is non-simple.
    Suppose that the statement holds for all positive words of word length less than $L$.
    By Lemma \ref{lemma concatenation}, $w$ can be written as the concatenation $w = u \cdot v$ of an arc $u$ and a curve $v$ satisfy $i(u) + i(v) + 1 \le i(w)$. Note that both $u$ and $v$ are positive words and of word length less than $L$. By the induction hypothesis, $i(u)\ge \alpha(u)$ and $i(u)\ge \beta(u)$. Using \cite{MR2904905}[Proposition 4.6] for the curve $v$, $i(v)\ge \alpha(v)$ and $i(v)\ge \beta(v)$. Therefore,
    \[
    i(u)+i(v)\ge\alpha(u)-1+\alpha(v)-1.
    \]
    Hence,
    \[
    i(w)\ge i(u)+i(v)+1=\alpha(u)+\alpha(v)-1=\alpha(w)-1.
    \]
    The $\beta$ inequality is proved in the same way.
\end{proof}

\begin{proof}[Proof of Theorem \ref{theorem lower bound intersections}]
By Lemma \ref{lemma positive word}, there is a positive word $w'$ of the same word length as $w$ such that $\alpha(w') = \alpha(w), \beta(w') = \beta(w)$, and $i(w')\le i(w)$. By Proposition \ref{prop intersection and alpha}, 
\[
i(w')\ge \max \{\alpha(w)-1,\beta(w)-1\}.
\]
Since $\alpha(w)+\beta(w)=L$, it follows that
\[i(w)\ge \frac{L}{2}-1\]
if $L$ is even and 
\[i(w)\ge \frac{L+1}{2}-1=\frac{L-1}{2}\]
if $L$ is odd.
\end{proof} 
We now show that the lower bounds are attainable. 
\begin{lemma}\label{lemma compute intersection some words}
For all $n\in \mathbb{N}$, 
\[i(1(BA)^n2) = n, \quad i(1(bA)^n3)=n^2+2n,\]
\[i(1(BA)^nB1) = n, \quad i(3(bA)^nb3)=n^2+3n+1.\]
\end{lemma}

\begin{proof}
We first compute $i(1(BA)^n2)$. We need to check all pair $(w_i,w_j)$ for $i=1,\dots,2n$ and $j=i,\dots, 2n+1$. 
We have $w_1=1B, w_{2n+1}=a2$, $w_{i}=bA$ for even $i$ and $w_i=aB$ for odd $i\ge3$.
If $i$ and $j$ have different parity, the pair $(w_i,w_j)$ is a non-intersecting decidable. For odd $i\ge3$, the pair $(w_1,w_i)=(1B,aB)$ is undecidable. We forward them until reaching $(w_\cdot, w_{2n+1})=(aB,a2)$, obtaining one intersection point. Note that while forwarding them, we mark all pairs $(w_e, w_{e'})=(bA,bA)$ for even $e,e'$ and $(w_o,w_{o'})=(aB,aB)$ for odd $o,o' (o\ge3)$ (see Table \ref{1BAn2intersection}). 
In total, we have $n$ intersections that come from pairs $(w_1,w_o)$ for odd $3\le o\le 2n+1$. 
Thus, the self-intersection numbers of $1(BA)^n2$ is 
\[
i(1(BA)^n2)=\frac{(2n-1)-1}{2}+1=n.
\]

\begin{table}[h]
    \centering
    \begin{tabular}{|l|c|c|c|c|c|c|c|} 
        \hline
&$w_1$&$w_2$&$w_3$&$w_4$&$w_{2n-1}$&$w_{2n}$&$w_{2n+1}$\\  
&$=1B$&$=bA$&$=aB$&$=bA$ &$=aB$&$=bA$&$=a2$\\\hline 
     $w_{1}=1B$& \cellcolor{gray!30} & 0& X& 0  & X&0&1\\\hline
        $w_{2}=bA$& \cellcolor{gray!30} & \cellcolor{gray!30} & 0 & X & 0&X&0\\\hline
        $w_{3}=aB$&  \cellcolor{gray!30} & \cellcolor{gray!30} & \cellcolor{gray!30} & 0 & X&0&1\\\hline
        $w_4=bA$& \cellcolor{gray!30} & \cellcolor{gray!30} & \cellcolor{gray!30} & \cellcolor{gray!30} &0&X&0\\\hline
        $w_{2n-1}=aB$&\cellcolor{gray!30}  & \cellcolor{gray!30} & \cellcolor{gray!30}  & \cellcolor{gray!30} & \cellcolor{gray!30}  &0&1  \\
        \hline
        $w_{2n}=bA$&\cellcolor{gray!30}  & \cellcolor{gray!30} & \cellcolor{gray!30}  & \cellcolor{gray!30} & \cellcolor{gray!30}  & \cellcolor{gray!30} &0   \\
        \hline
        $w_{2n+1}=a2$&\cellcolor{gray!30}  & \cellcolor{gray!30} & \cellcolor{gray!30}  & \cellcolor{gray!30} & \cellcolor{gray!30}  & \cellcolor{gray!30} & \cellcolor{gray!30}    \\
        \hline
    \end{tabular}
    \caption{Computing the self-intersection number of $1(BA)^n2$}
    \label{1BAn2intersection}
\end{table}

We do similarly to compute the self-intersection numbers of other arcs. 
Table \ref{1bAn3intersection} is for $1(bA)^n3$: its self-intersection number is 
\[
i(1(bA)^n3)=n\times 1+2\times 1 + 2\times 2 +\dots +2 \times n = n^2+2n.
\]

\begin{table}[h]
    \centering
    \begin{tabular}{|l|c|c|c|c|c|c|c|} 
        \hline
&$w_1$&$w_2$&$w_3$&$w_4$&$w_{2n-1}$&$w_{2n}$&$w_{2n+1}$\\  
&$=1b$&$=BA$&$=ab$&$=BA$ &$=ab$&$=BA$&$=a3$\\\hline 
     $w_{1}=1b$& \cellcolor{gray!30} & 1& X& 1 & X&1&1\\\hline
        $w_{2}=BA$& \cellcolor{gray!30} & \cellcolor{gray!30} & 1 & X & 1&X&1\\\hline
        $w_{3}=ab$&  \cellcolor{gray!30} & \cellcolor{gray!30} & \cellcolor{gray!30} & 1 & X&1&1\\\hline
        $w_4=BA$& \cellcolor{gray!30} & \cellcolor{gray!30} & \cellcolor{gray!30} & \cellcolor{gray!30} &1&X&1\\\hline
        $w_{2n-1}=ab$&\cellcolor{gray!30}  & \cellcolor{gray!30} & \cellcolor{gray!30}  & \cellcolor{gray!30} & \cellcolor{gray!30}  &1&1  \\
        \hline
        $w_{2n}=BA$&\cellcolor{gray!30}  & \cellcolor{gray!30} & \cellcolor{gray!30}  & \cellcolor{gray!30} & \cellcolor{gray!30}  & \cellcolor{gray!30} &1   \\
        \hline
        $w_{2n+1}=a3$&\cellcolor{gray!30}  & \cellcolor{gray!30} & \cellcolor{gray!30}  & \cellcolor{gray!30} & \cellcolor{gray!30}  & \cellcolor{gray!30} & \cellcolor{gray!30}    \\
        \hline
    \end{tabular}
    \caption{Computing the self-intersection number of $1(bA)^n3$}
    \label{1bAn3intersection}
\end{table}   

Table \ref{1BAnB1intersection} is for $1(BA)^nB1$: its self-intersection number is 
\[
i(1(BA)^nB1)=\frac{2n-2}{2}+1=n.
\]

\begin{table}[h]
    \centering
    \begin{tabular}{|l|c|c|c|c|c|c|c|c|} 
        \hline
&$w_1$&$w_2$&$w_3$&$w_4$&$w_{2n-1}$&$w_{2n}$&$w_{2n+1}$&$w_{2n+2}$\\  
&$=1B$&$=bA$&$=aB$&$=bA$ &$=aB$&$=bA$&$=aB$&$=b1$\\\hline 
     $w_{1}=1B$& \cellcolor{gray!30} & 0& X& 0  & X&0&X&0\\\hline
        $w_{2}=bA$ & \cellcolor{gray!30} & \cellcolor{gray!30} & 0 & X & 0&X&0&1\\\hline
        $w_{3}=aB$&  \cellcolor{gray!30} & \cellcolor{gray!30} & \cellcolor{gray!30} & 0 & X&0&X&0\\\hline
        $w_4=bA$& \cellcolor{gray!30} & \cellcolor{gray!30} & \cellcolor{gray!30} & \cellcolor{gray!30} &0&X&0&1\\\hline
        $w_{2n-1}=aB$&\cellcolor{gray!30}  & \cellcolor{gray!30} & \cellcolor{gray!30}  & \cellcolor{gray!30} & \cellcolor{gray!30}  &0&X&0  \\
        \hline
        $w_{2n}=bA$&\cellcolor{gray!30}  & \cellcolor{gray!30} & \cellcolor{gray!30}  & \cellcolor{gray!30} & \cellcolor{gray!30}  & \cellcolor{gray!30} &0&1   \\
        \hline
        $w_{2n+1}=aB$&\cellcolor{gray!30}  & \cellcolor{gray!30} & \cellcolor{gray!30}  & \cellcolor{gray!30} & \cellcolor{gray!30}  & \cellcolor{gray!30} & \cellcolor{gray!30} &0   \\
        \hline
        $w_{2n+2}=b1$&\cellcolor{gray!30}  & \cellcolor{gray!30} & \cellcolor{gray!30}  & \cellcolor{gray!30} & \cellcolor{gray!30}  & \cellcolor{gray!30} & \cellcolor{gray!30} &\cellcolor{gray!30}    \\
        \hline
    \end{tabular}
    \caption{Computing the self-intersection number of $1(BA)^nB1$}
    \label{1BAnB1intersection}
\end{table} 

Table \ref{3bAnb3intersection} is for $3(bA)^nb3$: its self-intersection number is 
\[
i(3(bA)^nb3)= n\times 1+ 1+ 3+5+7+\dots+(2n+1)= n+(n+1)^2=n^2+3n+1.
\]

\begin{table}[h]
    \centering
    \begin{tabular}{|l|c|c|c|c|c|c|c|c|} 
        \hline
&$w_1$&$w_2$&$w_3$&$w_4$&$w_{2n-1}$&$w_{2n}$&$w_{2n+1}$&$w_{2n+2}$\\  
&$=3b$&$=BA$&$=ab$&$=BA$ &$=ab$&$=BA$&$=ab$&$=B3$\\\hline 
     $w_{1}=3b$& \cellcolor{gray!30} & 1& X& 1 & X&1&X&1\\\hline
        $w_{2}=BA$& \cellcolor{gray!30} & \cellcolor{gray!30} & 1 & X & 1&X&1&1\\\hline
        $w_{3}=ab$&  \cellcolor{gray!30} & \cellcolor{gray!30} & \cellcolor{gray!30} & 1 & X&1&X &1\\\hline
        $w_4=BA$& \cellcolor{gray!30} & \cellcolor{gray!30} & \cellcolor{gray!30} & \cellcolor{gray!30} &1&X&1&1\\\hline
        $w_{2n-1}=ab$&\cellcolor{gray!30}  & \cellcolor{gray!30} & \cellcolor{gray!30}  & \cellcolor{gray!30} & \cellcolor{gray!30}  &1&X&1  \\
        \hline
        $w_{2n}=BA$&\cellcolor{gray!30}  & \cellcolor{gray!30} & \cellcolor{gray!30}  & \cellcolor{gray!30} & \cellcolor{gray!30}  & \cellcolor{gray!30} &1 &1  \\
        \hline
        $w_{2n+1}=ab$&\cellcolor{gray!30}  & \cellcolor{gray!30} & \cellcolor{gray!30}  & \cellcolor{gray!30} & \cellcolor{gray!30}  & \cellcolor{gray!30} & \cellcolor{gray!30}&1    \\
        \hline
        $w_{2n+2}=B3$&\cellcolor{gray!30}  & \cellcolor{gray!30} & \cellcolor{gray!30}  & \cellcolor{gray!30} & \cellcolor{gray!30}  & \cellcolor{gray!30} & \cellcolor{gray!30} & \cellcolor{gray!30}   \\\hline
    \end{tabular}
    \caption{Computing the self-intersection number of $3(bA)^nb3$}
    \label{3bAnb3intersection}
\end{table}   

\end{proof}

\begin{remark}
    We also attempted to obtain a sharp upper bound for the self-intersections of $w=n_1x_1\dots x_Ln_2$. However, the method in \cite[Section 2.2]{MR2904905} does not apply because we lack cyclic permutations for arcs. Therefore, our current upper bound is derived from Theorem \ref{algorithm}. The self-intersection number of $w$ is bounded from above by the total number of pairs $(w_i,w_j)$ for $i=1,\dots,L$ and $j=i,\dots,L+1$, hence
    \[
    i(w)\le L+L-1+\dots+1 = \frac{L(L+1)}{2}.
    \] 
    Based on computational experiments (see Table \ref{table:min_max_intersections}), 
    we expect the maximal self-intersection number of $w$ to be $\frac{L^2}{4}+L$ if $L$ is even and to be $\frac{L^2-1}{4}+L$ if $L$ is odd, and they are attained by the self-intersection numbers of the arcs $1(bA)^n3$ and $3(bA)^nb3$ as in 
    Lemma \ref{lemma compute intersection some words}.
\end{remark}

\section{Low-lying arcs}\label{sec:lowlying}
We conclude this note with a discussion on the self-intersection numbers of low-lying arcs and the proof of Theorem \ref{theorem low lying}.   
An arc is called $k$-\textit{low-lying} if it wraps around the same cuff (a boundary component) consecutively at most $k$ times. 
In terms of words, for example, if the arc $w=n_1x_1\dots x_Ln_2$ satisfies the condition that for all $i\le L$, 
\[x_ix_{i+1}x_{i+2}x_{i+3}\notin \{aaaa,AAAA,bbbb,BBBB\} \quad \textup{if } i+3\le L, L\ge 4,\]
and 
\[
x_ix_{i+1}\dots x_{i+6}x_{i+7} \notin\{abababab,ABABABAB,babababa,BABABABA\}  \quad \textup{if } i+7\le L, L\ge 8,
\] 
then $w$ is $4$-low-lying. 

Our computations reveal that for any natural number \(N \leq 1000\), there exists a low-lying arc whose self-intersection number is \(N\) (refer to Table \ref{table_low_lying_with_i} for \(N \leq 89\)). Theorem \ref{theorem low lying} extends this result to all natural numbers \(N \in \mathbb{N}\), providing a comprehensive view of the behavior of low-lying arcs in this context.

\begin{remark}
The observation about self-intersection numbers of low-lying arcs is somewhat analogous to Zaremba's conjecture \cite{zaremba1972methode, bourgain2014zaremba}, which predicts the existence of some natural number $A > 1$ such that 
\[
D_A:= \left\{
d\in \mathbb{N} : \exists (b,d)=1 \text{ with } \dfrac{b}{d} \in R_A
\right\}
=\mathbb{N}
\]
where
\[
R_A := \left\{
\dfrac{b}{d} = [a_1,a_2,...,a_k] : 0 < b < d, (b, d) = 1 \text{ and } \forall i, a_i \le A \right\},
\]
and 
\[
[a_1,a_2,...,a_k] := \dfrac{1}{a_1+\dfrac{1}{a_2+\ddots+\dfrac{1}{a_k}}}. 
\]
In our case, geometrically, a fraction $\frac{b}{d}$ corresponds to a geodesic arc on a hyperbolic pair of pants with three cusps $\mathcal{P} = \mathbb{H} \slash \Gamma(2)$ (see Figure \ref{fig:fundamentaldomains}), where $\mathbb{H}$ denotes the Poincar\'e upper half plane, and 
$$\Gamma(2) = \left\langle \begin{bmatrix} 1&0\\2&1 \end{bmatrix},\begin{bmatrix} 1&2\\0&1 \end{bmatrix}\right\rangle.$$

\begin{figure}[h] \centering \begin{tikzpicture}[scale=3] 1 
\fill[gray, opacity=0.3] (-1,1.3) -- (-1,0) -- (0,0) -- (1,0) -- (1,1.3) -- cycle; \fill[white, opacity=1] (1,0) arc[start angle=0, end angle=180, radius=0.5]; \fill[white, opacity=1] (0,0) arc[start angle=0, end angle=180, radius=0.5]; \draw[thick, decoration={markings, mark=at position 0.5 with {\arrow{>>}}}, postaction={decorate}] (1,0) arc[start angle=0, end angle=180, radius=0.5];; \draw[thick, decoration={markings, mark=at position 0.5 with {\arrow{<<}}}, postaction={decorate}] (0,0) arc[start angle=0, end angle=180, radius=0.5];; 
\draw[thick, postaction={decorate,decoration={markings,mark=at position 0.5 with {\arrow{>}}}}] (-1,0) -- (-1,1.3); \draw[thick, postaction={decorate,decoration={markings,mark=at position 0.5 with {\arrow{>}}}}] (1,0) -- (1,1.3); \draw[thick,] (0,0) -- (0,1.3); \draw[thick,] (1/3,0) -- (1/3,1.3); 
\fill (0,0) circle (0.5pt) node[below] {$0$}; \fill (1,0) circle (0.5pt) node[below] {$1$}; \fill (-1,0) circle (0.5pt) node[below] {$-1$}; \fill (1/3,0) circle (0.5pt) node[below] {$\frac{b}{d}$}; 
\draw[thick, ->] (-1.3,0) -- (1.3,0); \end{tikzpicture} \caption{A fundamental domain for the hyperbolic pair of pants $\mathcal{P}$ of three cusps, and a lift of the orthogeodesic $\mu_{\frac{b}{d}}$.} \label{fig:fundamentaldomains} \end{figure}
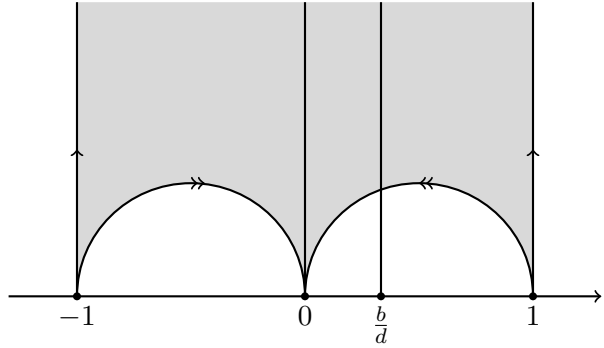 

\noindent
The set of lifts of cusps of $\mathcal{P}$ is the union of the orbits of cusps $0,1$ and $\infty$
\[
\Gamma(2)\cdot 0\cup \Gamma(2) \cdot 1 \cup \Gamma(2) \cdot \infty = \mathbb{Q}\cup \{\infty\}.
\]
An orthogeodesic on $\mathcal{P}$ is a geometric realization of an infinite arc on $\mathcal{P}$. 
Due to the symmetry of $\mathcal{P}$, we will consider only the set of orthogeodesics lifting to vertical lines from $\infty$ to fractions between $0$ and $1$ on the universal cover of $\mathcal{P}$ (identified to $\mathbb{H}$). We denote by $\mu_{\frac{b}{d}}$ the orthogeodesic on $\mathcal{P}$ lifting to the vertical line from $\infty$ to $\frac{b}{d}\in \mathbb{Q}\cap [0,1]$. 
Theorem A in \cite{series1985modular} implies that if $\frac{b}{d} = [a_1, a_2, \ldots, a_k]$, then the cutting sequence of $\mu_{\frac{b}{d}}$ is $R^{a_1 - 1}L^{a_2} \ldots X^{a_k-1}$, where $X = L$ (or $R$) if $k$ is even (or odd). Thus the "low" in the values of $a_i$ implies the "low" in the wrapping numbers of the orthogeodesic $\mu_{\frac{b}{d}}$ around each cuff consecutively. 
In these terms, the question about self-intersections of low-lying arcs can be stated as follows:
 
\begin{quote} Is there a natural number \( A > 1 \) such that for every \( N \in \mathbb{N}\), there exists a fraction \( \frac{b}{d} \in R_A \) for which the self-intersection number of \( \mu_{\frac{b}{d}} \) equals \( N \)?
\end{quote}
\end{remark}

In Theorem \ref{theorem low lying}, we give a positive answer to this question for $A=2$.

\begin{lemma}\label{lem:coverall} Let $$A_1:=\{(m+n+1)^2+2m+n|m,n \in \mathbb{N},m\geq 1\},$$
$$A_2:=\{(m+n)^2+2m+3n|m,n \in \mathbb{N},m\geq 1\},$$
$$A_3:=\{(n+4)^2-2|n \in \mathbb{N}\},$$
$$A_4:=\{n(n+3)|n \in \mathbb{N}\},$$
$$A_5:=\{n(n+3)+1|n \in \mathbb{N}\},$$ then 
    $$\mathbb{N}=\{2,7\}\cup \bigcup_{i=1}^5A_i.$$
    
\end{lemma}
\begin{proof} We observe that:
    $$A_1 = \{(m+n+2)^2-(n+3)|m,n\in\mathbb{N},m\geq 1\} = \bigcup_{j=4}^{\infty}\left(\bigcup_{i=-j}^{-3}\{j^2+i\}\right)\cup \{6\},$$

    $$A_2 = \{(m+n+1)^2+(n-1)|m,n\in\mathbb{N},m\geq 1\} = \bigcup_{j=4}^{\infty}\left(\bigcup_{i=-1}^{j-3}\{j^2+i\}\right)\cup \{3,8,9\},$$

     $$A_3 = \{(n+4)^2-2|n\in\mathbb{N}\} = \bigcup_{j=4}^{\infty}\{j^2-2\},$$

      $$A_4 = \{(n+1)^2+(n+1)-2|n\in\mathbb{N}\} = \bigcup_{j=4}^{\infty}\{j^2+j-2\}\cup\{0,4,10\},$$

      $$A_5 = \{(n+1)^2+(n+1)-1|n\in\mathbb{N}\} = \bigcup_{j=4}^{\infty}\{j^2+j-1\}\cup\{1,5,11\},$$

      Thus $$\bigcup_{i=1}^5A_i =\bigcup_{j=4}^{\infty}\left(\bigcup_{i=-j}^{j-1}\{j^2+i\}\right)\cup \{0,1,3,4,5,6,8,9,10,11\}=\mathbb{N} - \{2,7\}.$$
\end{proof}
 
\begin{proof}[Proof of Theorem \ref{theorem low lying}]
Consider the low-lying arcs $1(bA)^n(baa)^m2$.  
Table \ref{low-lying_arcs_intersection} presents how we count its self-intersection numbers, with 
\[
w_1=1b,w_2=BA,  w_{3}=ab,w_{4}=BA,\dots, w_{2n-1}=ab,w_{2n}=BA,w_{2n+1}=ab,
\]
\[
w_{1}':=w_{2n+2}=Ba,w_{2}':=w_{2n+3}=Aa,w_3':=w_{2n+4}=Ab,\dots,
w_{3m-2}'=Ba,w_{3m-1}'={Aa},w_{3m}'=A2.
\] 
The self-intersection number is 
\begin{align*} 
&(2n+2(n-1)+\dots+2\times 1)+(2m+2(m-1)+\dots+2\times 1-1)
+(n(2m+1)+n+(m+1)) \\ 
=\ &(n+m)^2+3n+2m.
\end{align*}

\begin{table}[h]
    \centering
    \begin{tabular}{|l|c|c|c|c|c|c|c|c|c|c|c|c|c|c|c|c|} 
        \hline
&$w_1$&$w_2$&$w_3$&$w_4$&$w_{2n-1}$&$w_{2n}$&$w_{2n+1}$&$w_1'$&$w_2'$&$w_3'$&$w_4'$&$w_5'$&$w_{6}'$
&$w_{3m-2}'$&$w_{3m-1}'$&$w_{3m}'$\\  
\hline 
     $w_{1}$& \cellcolor{gray!30} & \cellcolor{pink}1& \cellcolor{pink}X& \cellcolor{pink}1 & \cellcolor{pink}X&\cellcolor{pink}1&\cellcolor{pink}X
     &\cellcolor{celadon}0&\cellcolor{celadon}1&\cellcolor{celadon}X&\cellcolor{celadon}0&\cellcolor{celadon}1&\cellcolor{celadon}X&\cellcolor{celadon}0&\cellcolor{celadon}1&\cellcolor{celadon}1\\\hline
        $w_{2}$& \cellcolor{gray!30}& \cellcolor{gray!30}& \cellcolor{pink}1& \cellcolor{pink}X& \cellcolor{pink}1&\cellcolor{pink}X
        &\cellcolor{pink}1
        &\cellcolor{brightturquoise}1&X&X&0&X&X&0&X&X\\\hline
        $w_{3}$&  \cellcolor{gray!30}& \cellcolor{gray!30}& \cellcolor{gray!30}& \cellcolor{pink}1&\cellcolor{pink}X&\cellcolor{pink}1&\cellcolor{pink}X&
        \cellcolor{yellow}1&\cellcolor{yellow}1&\cellcolor{yellow}X&\cellcolor{yellow}1&\cellcolor{yellow}1&\cellcolor{yellow}X&\cellcolor{yellow}1&\cellcolor{yellow}1&\cellcolor{yellow}1\\\hline
        $w_4$& \cellcolor{gray!30} & \cellcolor{gray!30} & \cellcolor{gray!30} & \cellcolor{gray!30} &\cellcolor{pink}1&\cellcolor{pink}X&\cellcolor{pink}1
        &\cellcolor{brightturquoise}1 &X&X&0&X&X&0&X&X\\\hline
        $w_{2n-1}$&\cellcolor{gray!30}  & \cellcolor{gray!30} & \cellcolor{gray!30}  & \cellcolor{gray!30} & \cellcolor{gray!30}  &\cellcolor{pink}1&\cellcolor{pink}X
        &\cellcolor{yellow}1&\cellcolor{yellow}1&\cellcolor{yellow}X&\cellcolor{yellow}1&\cellcolor{yellow}1&\cellcolor{yellow}X&\cellcolor{yellow}1&\cellcolor{yellow}1&\cellcolor{yellow}1\\
        \hline $w_{2n}$&\cellcolor{gray!30}  & \cellcolor{gray!30} & \cellcolor{gray!30}  & \cellcolor{gray!30} & \cellcolor{gray!30}  & \cellcolor{gray!30} &\cellcolor{pink}1 
        &\cellcolor{brightturquoise}1&X&X&0&X&X&0&X&X\\
        \hline
        $w_{2n+1}$&\cellcolor{gray!30}  & \cellcolor{gray!30} & \cellcolor{gray!30}  & \cellcolor{gray!30} & \cellcolor{gray!30}  & \cellcolor{gray!30} &\cellcolor{gray!30}  &\cellcolor{yellow}1&\cellcolor{yellow}1&\cellcolor{yellow}X&\cellcolor{yellow}1&\cellcolor{yellow}1&\cellcolor{yellow}X&\cellcolor{yellow}1&\cellcolor{yellow}1&\cellcolor{yellow}1  \\
        \hline
        $w_1'$&\cellcolor{gray!30}  & \cellcolor{gray!30} & \cellcolor{gray!30}  & \cellcolor{gray!30} & \cellcolor{gray!30}  & \cellcolor{gray!30} & \cellcolor{gray!30}&\cellcolor{gray!30} &\cellcolor{babyblue}X&\cellcolor{babyblue}0&\cellcolor{babyblue}X&\cellcolor{babyblue}X&\cellcolor{babyblue}0&\cellcolor{babyblue}X&\cellcolor{babyblue}X&\cellcolor{babyblue}0    \\
        \hline 
        $w_{2}'$&\cellcolor{gray!30}  & \cellcolor{gray!30} & \cellcolor{gray!30}  & \cellcolor{gray!30} & \cellcolor{gray!30}  & \cellcolor{gray!30} & \cellcolor{gray!30} & \cellcolor{gray!30} & \cellcolor{gray!30}   
        &\cellcolor{babyblue}1&\cellcolor{babyblue}X&\cellcolor{babyblue}X&\cellcolor{babyblue}1&\cellcolor{babyblue}X&\cellcolor{babyblue}X&\cellcolor{babyblue}1\\\hline
        $w_{3}'$&\cellcolor{gray!30}  & \cellcolor{gray!30} & \cellcolor{gray!30}  & \cellcolor{gray!30} & \cellcolor{gray!30}  & \cellcolor{gray!30} & \cellcolor{gray!30} & \cellcolor{gray!30}  & \cellcolor{gray!30} & \cellcolor{gray!30}  
        &\cellcolor{babyblue}0&\cellcolor{babyblue}1&\cellcolor{babyblue}X&\cellcolor{babyblue}0&\cellcolor{babyblue}1&\cellcolor{babyblue}1\\\hline
        $w_{4}'$&\cellcolor{gray!30}  & \cellcolor{gray!30} & \cellcolor{gray!30}  & \cellcolor{gray!30} & \cellcolor{gray!30}  & \cellcolor{gray!30} & \cellcolor{gray!30} & \cellcolor{gray!30}   & \cellcolor{gray!30} & \cellcolor{gray!30} & \cellcolor{gray!30} 
        &\cellcolor{babyblue}X&\cellcolor{babyblue}0&\cellcolor{babyblue}X&\cellcolor{babyblue}X&\cellcolor{babyblue}0\\\hline
        $w_{5}'$&\cellcolor{gray!30}  & \cellcolor{gray!30} & \cellcolor{gray!30}  & \cellcolor{gray!30} & \cellcolor{gray!30}  & \cellcolor{gray!30} & \cellcolor{gray!30} & \cellcolor{gray!30}  & \cellcolor{gray!30} & \cellcolor{gray!30} & \cellcolor{gray!30} & \cellcolor{gray!30}  &\cellcolor{babyblue}1&\cellcolor{babyblue}X&\cellcolor{babyblue}X&\cellcolor{babyblue}1\\\hline
        $w_{6}'$&\cellcolor{gray!30}  & \cellcolor{gray!30} & \cellcolor{gray!30}  & \cellcolor{gray!30} & \cellcolor{gray!30}  & \cellcolor{gray!30} & \cellcolor{gray!30} & \cellcolor{gray!30}   & \cellcolor{gray!30} & \cellcolor{gray!30} & \cellcolor{gray!30} & \cellcolor{gray!30} & \cellcolor{gray!30} &
        \cellcolor{babyblue}0&\cellcolor{babyblue}1&\cellcolor{babyblue}1\\\hline
        $w_{3m-2}'$&\cellcolor{gray!30}  & \cellcolor{gray!30} & \cellcolor{gray!30}  & \cellcolor{gray!30} & \cellcolor{gray!30}  & \cellcolor{gray!30} & \cellcolor{gray!30} & \cellcolor{gray!30}  & \cellcolor{gray!30} & \cellcolor{gray!30} & \cellcolor{gray!30} & \cellcolor{gray!30} & \cellcolor{gray!30} & \cellcolor{gray!30} & 
        \cellcolor{babyblue}X&\cellcolor{babyblue}0\\\hline
        $w_{3m-1}'$&\cellcolor{gray!30}  & \cellcolor{gray!30} & \cellcolor{gray!30}  & \cellcolor{gray!30} & \cellcolor{gray!30}  & \cellcolor{gray!30} & \cellcolor{gray!30} & \cellcolor{gray!30}  & \cellcolor{gray!30} & \cellcolor{gray!30} & \cellcolor{gray!30} & \cellcolor{gray!30} & \cellcolor{gray!30} & \cellcolor{gray!30} & \cellcolor{gray!30}  &\cellcolor{babyblue}1\\\hline
        $w_{3m}'$&\cellcolor{gray!30}  & \cellcolor{gray!30} & \cellcolor{gray!30}  & \cellcolor{gray!30} & \cellcolor{gray!30}  & \cellcolor{gray!30} & \cellcolor{gray!30} & \cellcolor{gray!30}   & \cellcolor{gray!30} & \cellcolor{gray!30} & \cellcolor{gray!30} & \cellcolor{gray!30} & \cellcolor{gray!30} & \cellcolor{gray!30} & \cellcolor{gray!30} & \cellcolor{gray!30} \\\hline
    \end{tabular}
    \caption{Computing the self-intersection number of $1(bA)^n(baa)^m2$}
    \label{low-lying_arcs_intersection}
\end{table}       
\noindent
Similarly, we can compute the self-intersection numbers of the other low-lying arcs as in Table \ref{table needed low lying}.

\begin{table}[h]
\centering
\begin{tabular}{|l|l|l|l|}
\hline
\textup{Low-lying arcs $w$} & \textup{\textbf{$i(w)$}} & \textup{Conditions}& \textup{Continued fraction of $\frac{b}{d}$} \\
\hline
$1(bA)^n bab(ABA)^{m}2$ & $(m+n+1)^2 + 2m + n$ & $\forall m,n \in \mathbb{N}, m \geq 1$ & $[\underbrace{2, \ldots, 2}_{2n \text{ times}},1,2,1,1,\underbrace{2, \ldots, 2}_{2m-1 \text{ times}},1]$\\
\hline
$1(bA)^n(baa)^m2$ & $(m+n)^2 + 2m + 3n$ & $\forall m,n \in \mathbb{N}, m \geq 1$& $[\underbrace{2, \ldots, 2}_{2n \text{ times}}, 1, 1, \underbrace{2, \ldots, 2}_{2m-1 \text{ times}}, 1]$ \\
\hline
$1bABabaBA(bA)^n3$ & $(n+4)^2 - 2$ & $\forall n \in \mathbb{N}$ & $[2,1,1,1,1,2,1,1,1,1, \underbrace{2, \ldots, 2}_{2n \text{ times}},1]$\\
\hline
$1(bA)^n b1$ & $n(n+3)$ & $\forall n \in \mathbb{N}$& $[\underbrace{2, \ldots, 2}_{2n \text{ times}}, 1, 1]$ \\
\hline
$1(bA)^n ba2$ & $n(n+3) + 1$ & $\forall n \in \mathbb{N}$ & $[\underbrace{2, \ldots, 2}_{2n \text{ times}}, 1, 1,1]$\\
\hline
$1bA2$ & $2$ & & $[2, 1, 1]$\\
\hline
$1bAba3$ & $7$ & & $[2, 2, 1, 1, 1, 1]$\\
\hline
\end{tabular}
\caption{Some low-lying arcs and their corresponding self-intersection numbers}
\label{table needed low lying}
\end{table} 

\noindent
By Lemma \ref{lem:coverall}, the spectrum of self-intersection numbers of the 2-low-lying arcs in Table \ref{table needed low lying} covers all natural numbers.
\end{proof}

\begin{longtable}{|l|c|l|c|}
\caption{Examples of some low-lying words and their self-intersection numbers}\label{table_low_lying_with_i}
\\
\hline
\textbf{$w$} & \textbf{$i(w)$} & \textbf{$w$} & \textbf{$i(w)$} \\
\hline
\endfirsthead
\hline
\textbf{$w$} & \textbf{$i(w)$} & \textbf{$w$} & \textbf{$i(w)$} \\
\hline
\endhead
\hline
\endfoot
\hline
\endlastfoot
31 & 0 & 33 & 1 \\
3ab1 & 2 & 3aB1 & 3 \\
3aba3 & 4 & 3aBa3 & 5 \\
3abAb3 & 6 & 3abAb1 & 7 \\
3aBaB1 & 8 & 3abAbA2 & 9 \\
3AbAba3 & 10 & 3AbAbA3 & 11 \\
3abAbaB1 & 12 & 3abABaB1 & 13 \\
3ABaBaB1 & 14 & 3bAbAbA2 & 15 \\
3abAbaBa3 & 16 & 3abAbAbA3 & 17 \\
3AbAbAba3 & 18 & 3AbAbAbA3 & 19 \\
3abAbAbAb3 & 20 & 3aBabAbAb1 & 21 \\
3aBaBaBAb1 & 22 & 3ABaBaBaB1 & 23 \\
3bAbAbAbA2 & 24 & 3abAbAbAbA2 & 25 \\
3aBabAbAba3 & 26 & 3aBaBaBaBA3 & 27 \\
3AbAbAbAba3 & 28 & 3AbAbAbAbA3 & 29 \\
3abAbAbaBaB1 & 30 & 3abAbABaBaB1 & 31 \\
3aBabAbAbAb1 & 32 & 3aBaBaBaBAb1 & 33 \\
3ABaBaBaBaB1 & 34 & 3bAbAbAbAbA2 & 35 \\
3abAbAbaBaBa3 & 36 & 3abAbABaBaBa3 & 37 \\
3aBabAbAbAba3 & 38 & 3aBaBaBaBaBA3 & 39 \\
3AbAbAbAbAba3 & 40 & 3AbAbAbAbAbA3 & 41 \\
3abAbAbAbAbAb3 & 42 & 3abABabAbAbAb1 & 43 \\
3aBabABaBaBaB1 & 44 & 3AbABaBaBaBaB1 & 45 \\
3ABabAbAbAbAb1 & 46 & 3ABaBaBaBaBaB1 & 47 \\
3bAbAbAbAbAbA2 & 48 & 3abAbAbAbAbAbA2 & 49 \\
3abABabAbAbAba3 & 50 & 3aBabAbAbAbAbA3 & 51 \\
3AbAbAbAbAbaBa3 & 52 & 3ABaBaBaBaBaBA3 & 53 \\
3baBaBaBaBaBab3 & 54 & 3bAbAbAbAbAbAb3 & 55 \\
3abAbAbAbaBaBaB1 & 56 & 3abAbAbABaBaBaB1 & 57 \\
3abABabAbAbAbAb1 & 58 & 3aBabABaBaBaBaB1 & 59 \\
3AbABaBaBaBaBaB1 & 60 & 3ABabAbAbAbAbAb1 & 61 \\
3ABaBaBaBaBaBaB1 & 62 & 3bAbAbAbAbAbAbA2 & 63 \\
3abAbAbAbaBaBaBa3 & 64 & 3abAbAbABaBaBaBa3 & 65 \\
3abABabAbAbAbAba3 & 66 & 3aBabAbAbAbAbAbA3 & 67 \\
3AbAbAbAbAbAbaBa3 & 68 & 3ABaBaBaBaBaBaBA3 & 69 \\
3baBaBaBaBaBaBab3 & 70 & 3bAbAbAbAbAbAbAb3 & 71 \\
3abAbAbAbAbAbAbAb3 & 72 & 3abAbABabAbAbAbAb1 & 73 \\
3abABabABaBaBaBaB1 & 74 & 3aBabAbAbAbAbAbaB1 & 75 \\
3aBabABaBaBaBaBaB1 & 76 & 3AbABaBaBaBaBaBaB1 & 77 \\
3ABabAbAbAbAbAbAb1 & 78 & 3ABaBaBaBaBaBaBaB1 & 79 \\
3bAbAbAbAbAbAbAbA2 & 80 & 3abAbAbAbAbAbAbAbA2 & 81 \\
3abAbABabAbAbAbAba3 & 82 & 3abABabABaBaBaBaBa3 & 83 \\
3abABaBaBaBaBaBaBA3 & 84 & 3aBabAbAbAbAbAbAbA3 & 85 \\
3AbAbAbAbAbAbAbaBa3 & 86 & 3ABaBaBaBaBaBaBaBA3 & 87 \\
3baBaBaBaBaBaBaBab3 & 88 & 3bAbAbAbAbAbAbAbAb3 & 89 \\
\end{longtable}

\bibliographystyle{plain}
{\small
\bibliography{ref}}
\end{document}